%% file: paper.tex
\documentclass[10 pt,reqno]{amsart}
\usepackage{amsmath,amssymb,amsthm}
\usepackage{enumerate}
\usepackage{cite}
\usepackage[hidelinks]{hyperref}
\usepackage{tikz}
\usepackage{mathrsfs}
\usepackage[all,cmtip]{xy}
\include{paperdef}

\newtheorem*{claim*}{Claim}
\newtheorem{secLemma}{Lemma}[section]
\newtheorem{subsecLemma}{Lemma}[subsection]
\newtheorem*{Lemma*}{Lemma}
\newtheorem{secTheorem}[secLemma]{Theorem}
\newtheorem{subsecTheorem}[subsecLemma]{Theorem}
\newtheorem*{theorem*}{Theorem}
\newtheorem{secProp}[secLemma]{Proposition}
\newtheorem{subsecProp}[subsecLemma]{Proposition}
\newtheorem*{Prop*}{Proposition}
\newtheorem{secCor}[secLemma]{Corollary}
\newtheorem{subsecCor}[subsecLemma]{Corollary}
\newtheorem*{Cor*}{Corollary}
\newtheorem*{question*}{Question}

\theoremstyle{definition}
\newtheorem{secDef}[secLemma]{Definition}
\newtheorem{subsecDef}[subsecLemma]{Definition}
\newtheorem*{Def*}{Definition}
\newtheorem{secObs}[secLemma]{Observation}

\newtheorem*{Obs*}{Observation}

\newtheorem*{Fact*}{Fact}

\newtheorem{subsecRemark}[subsecLemma]{Remark}
\newtheorem*{Remark*}{Remark}

\newtheorem*{Notation*}{Notation}
\numberwithin{equation}{section}

\begin{document}
\title{Dilations of CP-maps Commuting According to a Graph}
\author{Alexander Vernik}
\date{November 5, 2014}
\address[Alexander Vernik]{Department of Mathematics, Faculty of Natural Sciences, Ben-Gurion University of the Negev}
\email{vernika@math.bgu.ac.il}
\maketitle
\begin{abstract}
	We study dilations of finite tuples of normal, completely positive and completely contractive maps (which we call CP-maps) acting on a von Neumann algebra, and commuting according to a graph $\gr{G}$. We show that if $\gr{G}$ is acyclic, then a tuple commuting according to it has a simultaneous *-endomorphic dilation, which also commutes according to $\gr{G}$. Conversely, if $\gr{G}$ has a cycle, we exhibit an example of a tuple of CP-maps commuting according to $\gr{G}$, which does not have an *-endomorphic dilation commuting according to $\gr{G}$. To achieve these results we use dilation theory of representations of subproduct systems, as introduced and studied by Shalit and Solel. In the course of our investigations we also prove some results about those kinds of subproduct systems which arise from CP-maps commuting according to to a graph.
\end{abstract}
\maketitle
\section{Introduction} \label{introduction}
	\subsection{Background}
	
		In the classic dilation theory of contractions on Hilbert spaces, the most basic theorem is Sz.-Nagy's dilation theorem \cite{Nagy1953}, which asserts that every contraction has a unitary dilation. Specifically, if $\hilb{H}$ is a Hilbert space, and $T\in\bdd{H}$ a contraction, then there is a Hilbert space $\hilb{K}\supseteq{H}$ and a unitary $U\in\bdd{K}$, such that for any $n\in\mathbb{N}$, 
		\[
		P_{\hilb{H}}\restr{U^n}{\hilb{H}} = T^n.
		\]
		In \cite{Ando1963}, Ando showed that one can obtain this result simultaneously for two commuting contractions. That is, if $S,T\in\bdd{H}$ commuting contractions, then there is a Hilbert space $\hilb{K}\supseteq{H}$ and unitaries $V,U\in\bdd{K}$, such that for all $n,m\in\mathbb{N}$,
		\[
		P_{\hilb{H}}\restr{V^nU^m}{\hilb{H}} = S^nT^m.
		\]
		For three contractions, the famous example of Parrott \cite{parrott1970} shows that this may not hold. Parrott exhibited three commuting contractions which do not have a commuting unitary dilation.
		
		A natural question is, given a specific set of commutation relations, will a tuple of contractions which satisfy these relations necessarily have a unitary dilation, which will also satisfy the same relations? To properly state this, we encode commutation relations in a (finite) graph in the following way.
		
		Let $\gr{G}=(V,E)$ be a finite graph. Let $\mathbf{F} = \set{f_v:v\in V}$ be a tuple of maps $f_v:X\rightarrow X$ on some set $X$. We say $\mathbf{F}$ \emph{commutes according to $\gr{G}$} if whenever $\set{u,v}\in E$, $f_u\circ f_v = f_v\circ f_u$ holds.
		
		It is also worthwhile to define precisely what we mean by a dilation of a tuple of contractions in this situation.
		For a tuple $\mathbf{T} = \set{T_v:v\in V}\subseteq\bdd{H}$ of contractions, we say that $\mathbf{U} = \set{U_v:v\in V}\subseteq\bdd{K}$ is a dilation of $\mathbf{T}$ if $\hilb{H}\subseteq\hilb{K}$, and for any choice $v_1\dots v_n\in V$, we have that
		\[
		P_{\hilb{H}} U_{v_1}\dots \restr{U_{v_n}}{\hilb{H}} = T_{v_1}\dots T_{v_n}.
		\]
		
		So the question can be formulated thusly: 
		\begin{question*}
			What conditions on $\gr{G}$ ensure that every tuple of contractions commuting according to it will have a unitary dilation, which will also commute according to $\gr{G}$? 
		\end{question*}
			Clearly, Sz.-Nagy's and Ando's theorems indicate that this is ensured when $\gr{G}$ is either one point, or two points connected by an edge, and Parrott's example states that this is not true for a clique of size 3. In \cite{Opvela2006} Opela answers this with the following theorem:
		\begin{subsecTheorem}[Opela]
			Let $\gr{G}$ be a finite graph. If $\gr{G}$ is acyclic, then every tuple of contractions $\mathbf{T}\subseteq\bdd{H}$ which commutes according to $\gr{G}$ has a unitary dilation, which commutes according to $\gr{G}$. Conversely, if $\gr{G}$ has a cycle, then there is a Hilbert space $\hilb{H}$ and a tuple of contractions on it, which does not have a unitary dilation commuting according to $\gr{G}$.
		\end{subsecTheorem}		
			We next turn our attention to dilation theory of CP-maps, which is what this paper is concerned with. A CP-map is a completely contractive, completely positive, normal map $\varphi:\vn{M}\rightarrow\vn{M}$, where $\vn{M}$ is a von Neumann algebra. We will always consider our von Neumann algebras as concrete subalgebras of $\bdd{H}$. We concern ourselves with *-endomorphic dilations of such maps, which are called E-dilations.
			
			Results analogous to Sz.-Nagy's and Ando's theorem are known: For a single CP-map, it was shown that an E-dilation always exists. For $\vn{M} = \bdd{H}$, this was first shown by Bhat in \cite{Bhat1996}
			and then by SeLegue in his PHD thesis \cite{Selegue1997}. For a general $\vn{M}$, this was shown by Bhat and Skeide \cite{BhatSkeide2000}, and by Muhly and Solel \cite{Muhly2002}. 
			It is worth noting that these results were actually given in these papers for the continuous one parameter case, that is, it was shown that there exists an E-dilation for every CP-semigroup indexed by $\mathbb{R}_+$.
			
			For two commuting CP-maps, the existence of E-dilations was shown again by Bhat \cite{Bhat1998} for the case of $\vn{M} = \bdd{H}$, and later by Solel \cite{Solel2006} for arbitrary $\vn{M}$.
			An analogue of Parrott's example was also described by Shalit and Solel in \cite{Shalit2009} (an example with unital maps was given by Shalit and Skeide in \cite{ShalitSkeide2011}). In the continuous case, dilation theory of two parameter CP-semigroups was studied by Shalit in \cite{Shalit2008} and \cite{Shalit2011}.
			
			It is natural, then, to ask whether Opela's result has an analogue in this theory. That is the main motivation of this work, and is answered in the affirmitive in Section \ref{cp maps section}. To show this, we use the connections between the theory of CP-semigroups and subproduct systems and their representations, introduced by Shalit and Solel in \cite{Shalit2009}.
			
			A \emph{CP-semigroup} is a collection $\mathbf{\varphi} = \set{\varphi_{s}: s\in\sgp{S}}$ of CP-maps, where $\sgp{S}$ is a semigroup, and $\mathbf{\varphi}$ satisfies $\varphi_{st} = \varphi_s\circ\varphi_t$, for all $s,t\in\sgp{S}$. We say $\mathbf{\varphi}$ is indexed by $\sgp{S}$. In the non discrete case, such as $\sgp{S} = \mathbb{R}_+$, additional continuity conditions may be imposed. However, throughout this work we shall only consider discrete cases.
			
			Given a tuple of $k$ commuting CP-maps, it is easy to construct from it a CP-semigroup indexed by $\mathbb{N}^k$. Given a tuple which commutes according to some graph $\gr{G}$, one must first construct the indexing semigroup $\sgp{S}(\gr{G})$. $\sgp{S}(\gr{G})$ will be the free semigroup with generators $V(\gr{G})$ (that is, words over the alphabet $V(\gr{G})$), modulo the commutation relations encoded by $\gr{G}$. The precise construction is given in the next section, but the gist of it is that elements of $\sgp{S}(\gr{G})$ have different presentations as finite strings of "letters" from $V(\gr{G})$, which are governed by the commutation relations $\gr{G}$ imposes. For example, if $\gr{G}$ is the graph $a-b-c$, then a typical word in $\sgp{S}(\gr{G})$ will be $s = abcb = bacb = babc = bbac = \dots$.
			
			An E-dilation of a CP-semigroup $\varphi$, is a CP-semigroup $\theta$, consisting of $*$- endomorphisms, acting on a von Neumann algebra $\vn{R}\supseteq\vn{M}$, such that $\vn{M} = p\vn{R}p$ for a projection $p\in\vn{R}$. It is required to satisfy 
			\[
			p\theta_s(a)p = \varphi_s(pap)
			\]
			for all $a\in\vn{R}$ and $s\in\sgp{S}$. The results we have listed before, analoguous to Sz.-Nagy's and Ando's theorem, are thus stating in this notation that CP-semigroups indexed by $\mathbb{N},\mathbb{N}^2$ have E-dilations. This paper is mainly concerned with proving the same for $\sgp{S}(\gr{G})$, for an acyclic graph $\gr{G}$.
			
	\subsection{Subproduct systems}
	
		Subproduct systems were introduced and studied by Shalit and Solel \cite{Shalit2009}. A subproduct system indexed by a semigroup $\sgp{S}$, is a collection of $C^*$-correspondences $\set{\sps{X}_s: s\in\sgp{S}}$, such that $\sps{X}_{st}$ embeds in $\sps{X}_s\otimes \sps{X}_t$, for all $s,t$, and these embeddings satisfy an associativity condition. Specifically, there are co-isometric bimodule maps, called the product maps, $U_{s,t}:\sps{X}_s\otimes \sps{X}_t \rightarrow \sps{X}_{st}$, which are associative in the sense that $U_{st,r}\circ U_{s,t}\otimes I = U_{s,tr}\circ I\otimes U_{t,r}$ for all $s,t,r$. The fibers $\sps{X}_s$ are $C^*$-correspondences (or $W^*$-correspondences, as would be the case in this paper) over $\sps{X}_e$ (with $e$ the neutral element in $\sgp{S}$), which is required to be a $C^*$ (or von Neumann) algebra, and the product maps $U_{s,e},U_{e,s}$ are taken to be the (left and right) action of $\sps{X}_e$.

		If $U_s$ are unitary, then this is a product system. Product systems (of Hilbert spaces) were first introduced by Arveson in his study and classification of E-semigroups (indexed by $\mathbb{R}_+$) on $\bdd{H}$ (See Arveson's monograph \cite{Arveson2003} and Skeide's survey paper \cite{Skeide2008}). Product systems were used by Bhat and Skeide \cite{BhatSkeide2000} and by Muhly and Solel \cite{Muhly2002} to prove that CP-semigroups indexed by $\mathbb{R}_+$ (and by $\mathbb{N}$) have E-dilations. They were also used by Solel \cite{Solel2006} to construct E-dilations for CP-semigroups indexed by $\mathbb{N}^2$.
			
		In \cite{Shalit2009}, Shalit and Solel associated to each CP-semigroup a subproduct system, coupled with a representation, which they called the Arveson-Stinespring subproduct system, with the identity representation. From this construction it is possible to retrieve the original CP-semigroup, and furthermore, isometric dilations of the identity representation correspond to E-dilations of the original CP-semigroup. This correspondence allows one to translate questions about CP-semigroups into questions about subproduct systems and their representation.
			
		It is worth noting that subproduct system were introduced independently, under the name \emph{inclusion systems}, in the paper by Bhat and Mukherjee \cite{BhatMukherjee2010}, simultaneously to Shalit and Solel's work.		
	\subsection{Overview of the paper}
		Section \ref{embedding section} is concerned with embedding subproduct systems in product systems. Specifically, we deal with subproduct systems indexed by a semigroup $\sgp{S}(\gr{G})$, which is associated to a graph $\gr{G}$. If $\gr{G}$ is without 3-cycles, we prove that any such subproduct system may be embedded in a product system. In particular this is true for acyclic graphs, which are our main concern.
		
		The construction itself is straightforward. The main difficulty is in showing that the product maps we define are associative. This turns out to be true because of combinatorial reasons, and the proof is given in full detail in the appendix.
		
		In Section \ref{dilation section}, we deal with product systems and their representations. We define the isometric dilation property. A semigroup $\sgp{S}$ has the isometric dilation property if any representation of a product system indexed by it has an isometric dilation. It is known that $\mathbb{N}$ has this property (this follows from a result in \cite{Muhly2002}), and in \cite{Solel2006} it is also shown for $\mathbb{N}^2$. We use these results to prove that if $\gr{G}$ is acyclic, then $\sgp{S}(\gr{G})$ has the isometric dilation property.
		
		Section \ref{cp maps section} ties these two results together into a result about CP-semigroups. Namely, we show that any CP-semigroup indexed by $\sgp{S}(\gr{G})$, where $\gr{G}$ is acyclic, has and E-dilation. Roughly, one takes a CP-semigroup indexed by $\sgp{S}(\gr{G})$ and constructs the Arveson-Stinespring subproduct system $\sps{X}$ and identity representation $T$, associated to $\varphi$. $\sps{X}$ is then embedded into a product system $\sps{Y}$ via a co-isometric morphism $V:\sps{Y}\rightarrow \sps{X}$. It then follows that $T\circ V$ is in fact a dilation of $T$. Using the result of Section \ref{dilation section}, one is able to further dilate it to an isometric representation, which in turn yields an E-dilation of $\varphi$.
		
		In the converse direction, we construct an example, for any graph $\gr{G}$ with a cycle, of a CP-semigroup indexed by $\sgp{S}(\gr{G})$ which does not have an E-dilation. We follow one of the ways Shalit and Solel constructed an example of three commuting CP-maps with no simultaneous E-dilations in \cite{Shalit2009}. They used Parrott's example to construct a product system (indexed by $\mathbb{N}^3$) and representation which cannot be isometrically dilated, and concluded, by results that they have shown in the same paper, that the CP-semigroup associated to it cannot have an E-dilation. We do the same, only instead of Parrott's example, we use the example Opela constructed of a tuple of contractions commuting according to $\gr{G}$ (which we assume to have a cycle), which does not have a unitary dilation commuting according to $\gr{G}$. Actually, we need such a tuple to not have an isometric co-extension, but we show that the existence of such a tuple of contractions follows from Opela's example.
\section{Preliminaries} \label{preliminaries}
	\subsection{The free semigroup with relations in $\gr{G}$}
	
		Let $\gr{G}$ be a graph with vertices $V(\gr{G})$ and edges $E(\gr{G})$. Then we define $\sgp{R}$ to be the free unital semigroup over $V(\gr{G})$. That is, $\sgp{R}$ is the set of finite sequences of elements in $V(\gr{G})$ (words), with $e$ being the empty word, and the binary associative action given by concatenation.
		
		We define an equivalence relation on $\sgp{R}$ in the following way. First, for two words $r,s\in \sgp{R}$, we will say that $r\sim s$ if $r=r_0v_1v_2r_1$ and $s=r_0v_2v_1r_1$, where $\set{v_1,v_2}\in E(\gr{G})$, and $r_0,r_1$ are some words. That is, if we think of the edges in $\gr{G}$ as commutation relations between the vertices, then we are saying that $r\sim s$ if $r$ can be rewritten into $s$ by applying a commutation relation once. We extend the definition of $\sim$ by taking its reflexive transitive closure. What we get overall is that $r\sim s$ iff $r$ can be rewritten into $s$ by applying a finite sequence of commutation relations. This is easily seen to be an equivalence relation on $\sgp{R}$.
		\begin{subsecDef}
		$\sgp{S}(\gr{G}):= \sgp{R}/\sim$ is called \emph{the free semigroup with relations in $\gr{G}$}.
		\end{subsecDef}
		For convenience, we will also call elements of $\sgp{S}(\gr{G})$ words.
		Observe that $\sgp{S}(\gr{G})$ has a semigroup structure, since concatenation of two equivalence classes via the concatenation of their representatives is well defined. Also, the equivalence class of the empty word of $\sgp{S}(\gr{G})$ is a neutral element for $\sgp{S}(\gr{G})$. 
		
		We call $V(\gr{G})$ the \emph{alphabet} of $\sgp{S}(\gr{G})$, and its elements \emph{letters}. We will drop the distinction between letters and their equivalence classes with respect to $\sim$, as the equivalence class of a letter is a singleton. We say $s_1s_2\dots s_n$ is a \emph{presentation} of $s\in \sgp{S}$ if $s=s_1s_2\dots s_n$ and $s_1,s_2,\dots,s_n$ are in the alphabet of $\sgp{S}(\gr{G})$. We observe that every presentation of a word in $\sgp{S}(\gr{G})$ has the same number of generators, and thus there is a well defined notion of length in $\sgp{S}(\gr{G})$.
	
	\subsection{$C^*$ and $W^*$ correspondences}
	
		\label{CWcorr}
		Let $\cstar{A}$ be a $C^*$-algebra.
		\begin{subsecDef}
		A $C^*$\emph{-correspondence} $\corr{E}$ over $\cstar{A}$ is a right Hilbert $C^*$-module over $A$, equipped with an additional adjointable left action of $\cstar{A}$. That is, a $*$-homomorphism $\varphi_E:\cstar{A} \rightarrow \adj{E}$. If additionally, $\cstar{A}$ is a von Neumann Algebra, $\varphi_E$ is normal and $\corr{E}$ is self dual, we call $\corr{E}$ a \emph{$W^*$-correspondence} over $\cstar{A}$. $\corr{E}$ is called \emph{essential} if $\varphi_{\corr{E}}(\vn{A})\corr{E}$ is dense in $\corr{E}$.
		\end{subsecDef}
		More generally, one defines a $C^*$-correspondence from $\cstar{A}$ to $\cstar{B}$ to be a right Hilbert $C^*$-module over $\cstar{B}$ with a left adjointable action of $\cstar{A}$, but we will not need this notion.
		We will omit $\varphi_E$ when writing the left action on a correspondence, and simply write $\varphi_E(a)x$ as $ax$.
		
		We note that for $W^*$-correspondences, it is known (Proposition 3.8 in \cite{paschke1973inner}) that they are always conjugate spaces, and we call the weak$*$ topology this induces on them the \emph{$\sigma$-weak topology}.
		
		The following notion of a covariant representation of a $C^*$-correspondence was studied by Muhly and Solel in \cite{Muhly1998}.
		\begin{subsecDef}
		Let $\corr{E}$ be a $C^*$-correspondence over $\cstar{A}$, and $\hilb{H}$ be a Hilbert space. A pair $(\sigma,T)$ is called a \emph{covariant representation} if:
		\begin{enumerate}
		\item $T:\corr{E}\rightarrow \bdd{H}$ is a completely bounded linear map.
		\item $\sigma:\cstar{A}\rightarrow \bdd{H}$ is a non-degenerate $*$-homomorphism.
		\item $T(ax) = \sigma(a)T(x)$ and $T(xa) = T(x)\sigma(a)$ for all $a\in\cstar{A},x\in\corr{E}$.
		\end{enumerate}
		If $\cstar{A}$ is a von Neumann algebra and $\corr{E}$ is a $W^*$-correspondence, we also require $\sigma$ to be normal.
		$(\sigma,T)$ is called a \emph{completely contractive covariant representation} if $T$ is completely contractive. It is called \emph{isometic} if $T(x)^*T(y) = \sigma(\langle x,y\rangle )$ for all $x,y\in\corr{E}$.
		\end{subsecDef}
		We will denote completely contractive covariant representation as c.c. representations, following \cite{Shalit2009}. The exact meaning of the adverb \emph{completely} here involves endowing $\corr{E}$ with an operator space structure, and is explained when defined in \cite{Muhly1998}.
		
		Given a c.c representation $(\sigma,T)$ of $\corr{E}$ on $\hilb{H}$, an important construction is the Hilbert space $\corr{E}\otimes_{\sigma}\hilb{H}$, which is defined to be the Hausdorff completion of $\corr{E}\otimes_{alg}\hilb{H}$, with respect to the pre-inner product defined by
		\[[x\otimes h,y\otimes g]:= \langle h,\sigma( \langle x,y \rangle )g \rangle .\]
		On $\corr{E}\otimes_{\sigma} \hilb{H}$ one defines $\tild{T}:\corr{E}\otimes_{\sigma} \hilb{H}\rightarrow \hilb{H}$ by 
		\[\tild{T}(x\otimes h):= T(x)h.\]
		In \cite[Lemma 3.5]{Muhly1998} it is shown that $(\sigma,T)$ is completely bounded (isometric/completely contractive) iff $\tild{T}$ is bounded (isometric/contractive).
		\begin{subsecLemma} \cite[Lemma 2.16]{Muhly2002} \label{bijection lemma}
		Let $\corr{E}$ be a $W^*$-correspondence over $\vn{M}$, with a normal representation $\sigma:\vn{M}\rightarrow\bdd{H}$.
		Then there is a bijection between c.c representation $(\sigma,T)$ of $\corr{E}$ and contractions  $\tilde{T}:\corr{E}\otimes_\sigma\hilb{H}\rightarrow\hilb{H}$ satisfying $\tilde{T}(ax\otimes h) = \sigma(a)\tilde{R}(x\otimes h)$ for all $a\in\vn{M},x\in\corr{E},h\in\hilb{H}$. This bijection is given by $T\mapsto\tilde{T}$.
		\end{subsecLemma}
		\comm{
		\begin{subsecRemark}
		\label{continuityRemark}
		We note that a c.c. representation $(\sigma,T)$ of a $W^*$-correspondence $\corr{E}$ (over $\vn{M}$) on a Hilbert space $\hilb{H}$ is continuous, w.r.t the $\sigma$-weak topology on $\corr{E}$, and the WOT on $\bdd{H}$. This follows easily from the fact that $\tild{T}$ is bounded.
		\end{subsecRemark} 
		In fact it is true that $T$ is $\sigma$-weak to ultraweak continuous, but we will not need this.
		}
		
		Given $C^*$-correspondences $\corr{E}$ and $\corr{F}$ over the same $C^*$-algebra $\cstar{A}$, we define $\corr{E}\otimes\corr{F}$ to be the interior tensor product (as in Chapter 4 in Lance's book \cite{Lance1995}), with inner product given by $ \langle x_0\otimes y_0,x_1\otimes y_1 \rangle  =  \langle y_0, \langle x_0,x_1 \rangle y_1 \rangle $. We give it a left adjointable action by $a\cdot x\otimes y:= (ax)\otimes y$, and this makes it a $C^*$-correspondence. If $\corr{E},\corr{F}$ are $W^*$-correspondences, we take the self-dual extension the interior tensor product instead (see section 3 in \cite{paschke1973inner}, specifically Theorem 3.2). The left adjointable action is defined in the same way, noting that by Theorem 3.6 in \cite{paschke1973inner}, it really is enough to define the action on pure tensors. We note that the construction of $\corr{E}\otimes_{\sigma}\hilb{H}$, where $(\sigma,T)$ was a representation of a correspondence $\corr{E}$, is in fact this interior tensor product, when $\hilb{H}$ is thought of as a correspondence (from $\vn{A}$ to $\mathbb{C}$), with the left action given by $\sigma$.
		\begin{subsecRemark}\label{pure tensor remark}
		In \cite[Proposition 3.6]{paschke1973inner}, Paschke showed that if $X,Y$ are inner product $\vn{M}$-modules, then any bounded module map $T:X\rightarrow Y$ has a unique extension to the self-dual completion. This implies that in the case of two $W^*$-correspondences $\corr{E}$ and $\corr{F}$, any bounded module map from $\corr{E}\otimes_{alg} \corr{F}$ can be extended uniquely to a map from $\corr{E}\otimes\corr{F}$, and any module map from $\corr{E}\otimes \corr{F}$ is determined by its action on the pure tensors.
		\end{subsecRemark}
		
		\begin{subsecProp}\label{tensorProp}
		Let $\corr{E},\corr{E'},\corr{F},\corr{F'}$ be $W^*$-correspondences over the same von Neumann algebra $\vn{M}$, and let $\alpha\in\adj{\corr{F},\corr{F'}}$,  $\beta\in\adj{\corr{E},\corr{E'}}$ be contractive bimodule maps. Then there exist unique, contractive, adjointable bimodule maps $I\otimes\alpha$,$\beta\otimes I$, such that for pure tensors $e\otimes f$, they satisfy: 
		\begin{equation}
		\label{itensoralpha}
		I\otimes\alpha(e\otimes f) = e\otimes\alpha f.
		\end{equation}
		\begin{equation}
		\label{betatensori}
		\beta\otimes I(e\otimes f) = \beta e\otimes f.
		\end{equation}
		\end{subsecProp}
		We omit the proof, as it is standard.
		
		As a conclusion, we see that by composing $(\beta\otimes I)\circ(I\otimes\alpha)$, we get a well defined adjointable bimodule map $\beta\otimes\alpha$ which on pure tensors satisfies $(\beta\otimes\alpha)(e\otimes f) = \beta e\otimes\alpha f$.
	
	\subsection{CP-maps, CP-semigroups and their dilations}
		
		Let $\vn{M}$ be a von Neumann algebra. We call $\varphi:\cstar{M}\rightarrow\cstar{M}$ a \emph{CP-map} if it is a completely positive, completely contractive and normal. Given a unital semigroup $\sgp{S}$, we say that $\Theta=\setind{\Theta}{s}{S}$ is a \emph{CP-semigroup} over $\vn{M}$ if for all $s,t\in \sgp{S}$, $\Theta_s:\vn{M}\rightarrow\vn{M}$ is CP, $\Theta_s\circ\Theta_t = \Theta_{st}$ and $\Theta_e = Id_{\vn{M}}$. We then say that $\Theta$ is \emph{indexed} by $\sgp{S}$. A CP-semigroup is called an E-semigroup if it consists of $*$-endomorphisms.
		
		We note that a single CP-map $\Theta$ generates a CP-semigroup indexed $\mathbb{N}$ by setting $\Theta_n:=\Theta^n$. Two commuting CP-maps $\varphi,\psi$ will generate a CP-semigroup indexed by $\mathbb{N}^2$ by setting $\Theta_{(n,m)}:=\varphi^n\psi^m$. More generally, suppose $\sgp{S}$ is the free semigroup with relations in a graph $\gr{G}$, and for each letter $t$ in the alphabet of $\sgp{S}$ we associate a CP-map $\Theta_t$, such that they commute according to $\gr{G}$. We may then define the CP-semigroup indexed by $\sgp{S}$ by setting $\Theta_s:=\Theta_{t_1}\dots \Theta_{t_k}$, where $s=t_1\dots t_k$ is some presentation of $s\in\sgp{S}$. This is well defined and a CP-semigroup exactly because we chose the CP-maps to commute according to $G$.
		
		We proceed to define a notion of dilation for CP-maps and CP-semigroups:
		\begin{subsecDef}
		A triple $(p,\vn{R},\alpha)$ is called an \emph{E-dilation} of a CP-semigroup $\Theta = \setind{\Theta}{s}{S}$ if $\vn{R}$ is a von Neumann algebra containing $\vn{M}$, $p$ a projection in $\vn{R}$ such that $\vn{M} = p\vn{R}p$, and $\alpha = \setind{\alpha}{s}{S}$ an E-semigroup over $\vn{R}$ such that 
		\[
		\Theta_s(pap) = p\alpha_s(a)p
		\]
		for all $s\in\sgp{S}$,$a\in\vn{R}$.
		\end{subsecDef}
		\begin{subsecDef}
		Let $\varphi:\vn{M}\rightarrow\vn{M}$ be CP-map. We say a triple $(p,\vn{R},\psi)$ is an \emph{E-dilation} of $\varphi$ if $(p,\vn{R},\Psi)$ is an $e$-dliation of $\Phi$, where $\Psi,\Phi$ are the CP-semigroups generated by $\psi,\varphi$ (indexed by $\mathbb{N}$).
		\end{subsecDef}
		We note that this notion of dilation is different from the Stinespring dilation \cite{Stinespring1955} of a CP-map.
		
		We will say that two commuting CP-maps have a \emph{simultaneous} E-dilation if the CP-semigroup they generate (indexed by $\mathbb{N}^2$) has an E-dilation. More generally, for a tuple of CP-maps commuting according to a graph $\gr{G}$, we will say they have a simultaneous E-dilation (w.r.t $\gr{G}$) if the CP-semigroup they generate, indexed by $\sgp{S}(\gr{G})$ as above, has an E-dilation.
		
		It is known that for a single CP-map, an E-dilation always exists. This was shown by Bhat in \cite{Bhat1996} for $\vn{M}=\bdd{H}$ (albeit for a notion of dilation which is slightly weaker than the one used here), and later by Muhly and Solel in \cite[Theorem 2.12]{Muhly2002} in its full generality.
		For two commuting CP-maps, it is also known that there is always a simultaneous E-dilation. This was shown by Bhat \cite{Bhat1998} for $\bdd{H}$ (again, for a weaker notion of dilation than the one we use), and later by Solel \cite{Solel2006}. For three commuting CP-maps there is a counter example, akin to Parrott's example in the classical dilation theory of contractions on Hilbert spaces, given by Shalit and Solel in \cite{Shalit2009}. In that paper, Shalit and Solel defined the notion of a subproduct system, and drew connections between their representation theory and the theory of CP-semigroups. We therefore proceed to define these notions.
	
	\subsection{Subproduct systems, product systems and their representations}
	
		Let $\vn{M}$ be a von Neumann algebra and $\sgp{S}$ a semigroup with a neutral element $e$.
		\begin{subsecDef}
		A \emph{subproduct system} over $\vn{M}$, indexed by $\sgp{S}$, is a collection $\sps{X}=\setind{X}{s}{S}$ of $W^*$-correspondences over $\vn{M}$ (which we will sometimes call \emph{fibers}), and \emph{product maps} $U_{s,t}$ that satisfy the following conditions:
		\begin{enumerate}
		\item $X_e = \vn{M}$.
		\item For every $s,t\in \sgp{S}$, $U_{s,t}:\sps{X}_s\otimes \sps{X}_t\rightarrow\sps{X}_{st}$ is a co-isometric $\vn{M}$-correspondence map.
		\item The maps $U_{e,s},U_{s,e}$ are given by the right and left actions of $\vn{M}$.
		\item The product maps compose associatively, in the sense that 
		\[U_{rs,t}(U_{r,s}\otimes Id_{\sps{X}_t}) = U_{r,st}(Id_{\sps{X}_r}\otimes U_{s,t}).\]
		\end{enumerate}
		If additionally, the product maps are isometric, we call $\sps{X}$ a \emph{product system}.
		\end{subsecDef}
		\begin{subsecRemark}
		It is an easy calculation to check that $U_{e,s},U_{s,e}$, if given by the left and right action of $\vn{M}$, are automatically isometries, and $U_{s,e}$ is also a unitary. Requiring $U_{e,s}$ to be a co-isometry is then simply requiring it to be onto, and that is equivalent to saying that $\sps{X_s}$ is essential. So in fact, with the way we defined a subproduct system (and consequently, a product system), we have implicitly required all the $W^*$-correspondences in it to be essential.
		\end{subsecRemark}
		\begin{subsecRemark}
		\label{uniqueremark}
		We note that that for any $s=s_1\dots s_n$, there exists a co-isometric map 
		\[
		U_{s_1,\dots,s_n}:\sps{X}_{s_1}\otimes\dots\otimes\sps{X}_{s_n}\rightarrow \sps{X}_s
		\] 
		given by composing product maps tensored with the identity, in some order. The associativity condition assures us that this map is unique, meaning that the order of compositions does not matter. We will continue using the notation $U_{s_1,\dots,s_n}$ for this unique map, and we will also sometimes drop the subscript completely from both this map, and the product maps, when there is no ambiguity about what they should be. That is, we will denote maps of the form $U_{s,t}$ or $U_{s_1,\dots s_n}$ simply as $U$ (or as $U^X$ when there is another subproduct system at play).
		\end{subsecRemark}

		We proceed to define representations and morphisms of subproduct systems.
		\begin{subsecDef}
		Let $\sps{X} = \setind{\sps{X}}{s}{S}$ be a subproduct system, and $\hilb{H}$ a Hilbert space. A \emph{representation} $T = \setind{T}{s}{S}$ of $\sps{X}$ on $\hilb{H}$ is a collection of maps that satisfy:
		\begin{enumerate}
		\item $(T_e,T_s)$ is a c.c. representation of $\sps{X}_s$ on $\hilb{H}$ for all $s\in\sgp{S}$.
		\item $T_{st}(U_{s,t}(x\otimes y)) = T_s(x)T_t(y)$ for all $s,t\in\sgp{S},x\in \sps{X}_s,y\in\sps{X}_t$.
		\end{enumerate}
		We usually denote $T_e=\sigma$. When more convenient, we will treat $T$ as a function $T:\sps{X}\rightarrow \bdd{H}$. We say $T$ is \emph{isometric} if $(\sigma,T_s)$ is isometric for all $s\in\sgp{S}$, and that $T$ is \emph{fully co-isometric} if $(\sigma,T_s)$ is fully co-isometric for all $s\in\sgp{S}$.
		\end{subsecDef}
		
		\begin{subsecDef}
		Let $X,Y$ be subproduct systems indexed by the same semigroup $\sgp{S}$, with product maps $\set{U^X_{s,t}}$ and $\set{U^Y_{s,t}}$, respectively. Assume also that $\vn{M}=\sps{X}_e=\sps{Y}_e$. A \emph{morphism} from $Y$ to $X$ is a collection $V=\setind{V}{s}{S}$ of co-isometric $\vn{M}$-correspondence maps $V_s:\sps{Y}_s\rightarrow \sps{X}_s$, where $V_e$ is the identity, and they satisfy:
		\[U^X_{s,t}\circ (V_s\otimes V_t) = V_{st}\circ U^Y_{s,t}\]
		for all $s,t\in\sgp{S}$.
		We call $V$ an \emph{isomorphism} if $V_s$ are isometric for all $s\in\sgp{S}$.
		\end{subsecDef}
		More generally, we may allow $\sps{X}_e$ and $\sps{Y}_e$ to be different, but isomorphic, algebras. In this case we require $V_e$ to be a $*$-isomorphism, and the correspondence maps to be bimodule maps with respect to $V_e$.
		
		If $\sps{X}_s$ is a closed subspace of $\sps{Y}_s$ for all $s\neq e$ (and $\sps{X}_e=\sps{Y}_e$), and the orthogonal projections $p_s:\sps{Y}_s\rightarrow \sps{X}_s$ constitute a morphism, we will say that $\sps{X}$ is a \emph{subproduct subsystem} of $\sps{Y}$. We remark that it follows easily that there is a morphism from $\sps{Y}$ to $\sps{X}$ if and only if $\sps{X}$ is isomorphic to a subsystem of $\sps{Y}$.
			
	\subsection{CP-semigroups and representations of subproduct systems}
		
		As said earlier, Shalit and Solel introduced the notion of subproduct systems as a tool in studying CP-semigroups and their dilation theory. To this end, in \cite{Shalit2009} they associated to each CP-semigroup a subproduct system coupled with a representation, which is in some sense unique. Although they stated this correspondence only for subsemigroups of $\mathbb{R}^k_+$, their construction, along with the theorems we will use in this paper, work for arbitrary semigroups (with a unit), with the proofs unchanged. We will briefly describe the construction and results, without going into full details.
		
		From now on we fix a von Neumann algebra $\vn{M}\subseteq\bdd{H}$, always thinking about $\vn{M}$ represented faithfully on $\hilb{H}$ with the identity representation.
		
		Given a CP-map $\Theta$ on $\vn{M}$, we may define the Hilbert space $\vn{M}\otimes_{\Theta} \hilb{H}$ in a similar way to what we did before, to be the Hausdorff completion of $\vn{M}\otimes_{alg} \hilb{H}$ with respect to the positive semi-definite sesquilinear form
		\[
		[x\otimes h,y\otimes g]:= \langle h,\Theta( \langle x,y \rangle )g \rangle.
		\]
		We define a representation $\pi_{\Theta}$ of $\vn{M}$ on $\vn{M}\otimes_{\Theta} \hilb{H}$ by
		\[
		\pi_{\Theta}(S)(T\otimes h) = ST\otimes h.
		\]
		This is in fact nothing but the Stinespring dilation of $\Theta$.
		We define a contractive linear map $W_{\Theta}:\hilb{H}\rightarrow \vn{M}\otimes_{\Theta} \hilb{H}$ by
		\[W_{\Theta}(h) = I\otimes h.\]
		We then define the $W^*$-correspondence
		\[
		X_{\Theta} = \mathcal{L}_{\vn{M}}(\hilb{H},\vn{M}\otimes_{\Theta} \hilb{H})
		\]
		where $\mathcal{L}_{\vn{M}}(\hilb{H},\vn{M}\otimes_{\Theta} \hilb{H})$ is the space of bounded operators $S\in B(H,\vn{M}\otimes_{\Theta} \hilb{H})$ which intertwine the identity representation and $\pi_{\Theta}$, i.e. $ST = \pi_{\Theta}(T)S$ for all $T\in\vn{M}$.
		$\sps{X}_{\Theta}$ is a $W^*$-correspondence over $\vn{M}'$, the commutant of $\vn{M}$ in $\bdd{H}$. The actions are given by
		\[S\cdot X:=(I\otimes S)\circ X,\text{  }X\cdot S := X\circ S,\]
		and the $\vn{M'}$ valued inner product is given by
		\[
		\langle X,Y\rangle := X^*\circ Y
		\]
		$\sps{X}_{\Theta}$ is called the \emph{Arveson-Stinespring correspondence associated with $\Theta$}. We also get a c.c. representation by taking $\sigma = Id_{\vn{M}'}$ and $T_{\Theta}(X) = W_{\Theta}^*X$. $(\sigma,T_{\Theta})$ is then called the \emph{identity representation} associated with $\Theta$.
		The full details of this construction are detailed in \cite{Muhly2002}, and as Shalit and Solel note when citing this in \cite{Shalit2009}, the relevant results in \cite{Muhly2002} hold for non unital CP-maps as well, with the proofs unchanged.
		
		We continue following Chapter 2 in \cite{Shalit2009} and define, given a CP-semigroup $\setind{\Theta}{s}{S}$, $\sps{X}_e = \vn{M}$ and $\sps{X}_s:=\sps{X}_{\Theta_s}$ for $s\neq e$. These will be the fibers of our subproduct system $\sps{X}$. Giving the definitions of the product maps would only serve to obfuscate this brief survey, so it suffices to say that their definitions and the reason they work are worked out in full detail in \cite{Shalit2009}.
		
		We may also define $T_s:= T_{\Theta_s}$ and $T_e = Id_{\vn{M}}$ and get a representation 
		$T$ of $\sps{X}$. This construction is called the \emph{Arveson-Stinespring subproduct system} associated with $\Theta$, and $T$ is the \emph{identity representation of $\Theta$}. The pair $(X,T)$ is denoted by $\Xi(\Theta)$. It satisfies the neat property
		\[
		\Theta_s(a) = \tild{T_s}(I_{\sps{X}_s}\otimes a)\tild{T_s}^*
		\]
		for all $a\in\vn{M}$. That is, one can recover $\Theta$ from $\Xi(\Theta)$. 
		\begin{subsecRemark}
		\label{remark1}
		It is also worth noting that given any subproduct system representation $T$, the RHS in the equality above will define a CP-semigroup, and that CP-semigroup is denoted by $\Sigma(\sps{X},T)$. Furthermore, if $T$ is isometric, then it is an E-semigroup, and conversely, if $(\sps{X},T)$ was obtained as the Arveson-Stinespring subproduct system of an E-semigroup, then $T$ is isometric \emph{and} $\sps{X}$ is a product system.
		
		Another useful fact, as shown in \cite[Theorem 2.6]{Shalit2009}, is the following: If $\Xi\circ\Sigma(\sps{X},T) = (\sps{Y},R)$, and $T_s$ is injective for all $s$, then $\sps{X}$ and $\sps{Y}$ are isomorphic, and $R$ is given via $T$ composed with this isomorphism.
		\end{subsecRemark}
		It is also true that any subproduct system arises as the Arveson-Stinespring subproduct system of some CP-semigroup, as one can always define an injective representation of $\sps{X}$ on the Fock space $\bigoplus_{s\in\sgp{S}}\sps{X}_s$ \cite[Section 2.3]{Shalit2009} by 
		\[
		T_s{x}(y):=U_{s,t}(x\otimes y),
		\]
		when $y\in \sps{X}_t$.
		
		We can finally define dilations of representations of subproduct systems, and discuss their connection to dilations of CP-semigroups.
		\begin{subsecDef}
		Let $\sps{X}$ and $\sps{Y}$ be subproduct systems over $\vn{M}$ indexed by the same semigroup $\sgp{S}$, and let $T,R$ be representations of $\sps{X},\sps{Y}$ on Hilbert spaces $\hilb{H},\hilb{K}$, respectively. $(\sps{Y},R,\hilb{K})$ is called a \emph{dilation} of $(\sps{X},T,\hilb{H})$ if
		\begin{enumerate}
		\item $\sps{X}$ is a subsystem of $\sps{Y}$,
		\item \hilb{H} is a subspace of $\hilb{K}$,
		\item For all $s\in\sgp{S}$, $\tild{R}^*_s\hilb{H}\subseteq \sps{X}_s\otimes \hilb{H}$ and $\restr{\tild{R}^*_s}{\hilb{H}} = \tild{T}^*_s$.
		\end{enumerate}
		$R$ is said to be an \emph{isometric} dilation if $R$ is an isometric representation.
		\end{subsecDef}
		The third item can be replaced with the following three conditions:
		\begin{enumerate}[(a)]
		{\setlength\itemindent{0.5cm}
		\item $\restr{R_e(\cdot)}{\hilb{H}} = P_{\hilb{H}}\restr{R_e(\cdot)}{\hilb{H}} = T_e(\cdot)$,
		\item $P_{\hilb{H}}\restr{\tild{R}_s}{\sps{X}_s\otimes\hilb{H}} = \tild{T}_s$ for all $s\in\sgp{S}$,
		\item $P_{\hilb{H}}\restr{\tild{R}_s}{\sps{Y}_s\otimes\hilb{K}\ominus\sps{X}_s\otimes\hilb{H}} = 0$ for all $s\in\sgp{S}$.
		}
		\end{enumerate}
		If we further restrict $\sps{X}=\sps{Y}$ to be product systems, then the definition coincides with the definition of a dilation of a product system given in \cite{Muhly2002}.
		The following lemma will be useful to us when we wish to verify a certain representation is a dilation of another, when the subproduct system is fixed:
		\begin{subsecLemma}
		\label{coexLemma}
		Let $T,R$ be two representations of the same subproduct system $\sps{X}$ on $\hilb{H},\hilb{K}$, respectively, with $\hilb{H}\subseteq\hilb{K}$. Then $(\sps{X},R,\hilb{K})$ is a dilation of $(\sps{X},T,\hilb{H})$ if and only if $R(x)$ is a co-extension of $T(x)$, for all $x\in\sps{X}$.
		\end{subsecLemma}
		\begin{proof}
		We begin by recalling that $R(x)$ is a co-extension of $T(x)$ if and only if 
		\begin{equation}
		\label{inCoexLemma}
		\begin{split}
		&P_{\hilb{H}}\restr{R(x)}{\hilb{H}} = T(x)\\
		&P_{\hilb{H}}\restr{R(x)}{\hilb{H^\perp}} = 0.
		\end{split}
		\end{equation}
		If $R$ is a dilation of $T$, then for all $x\in \sps{X}_s$ and $h,g\in\hilb{H}$, we have by (b): 
		\[
		 \langle R_s(x)h,g \rangle  =  \langle \tild{R}_s(x\otimes h),g \rangle  =  \langle \tild{T}_s(x\otimes h),g \rangle  =  \langle T_s(x)h,g \rangle 
		\]
		and by (c), now with $k\in \hilb{H}^\perp$
		\[
		 \langle R_s(x)k,g \rangle  =  \langle \tild{R}_s(x\otimes k),g \rangle  =  \langle 0,g \rangle  = 0
		\]
		In the other direction, if we assume (\ref{inCoexLemma}), we get (b) and (c) for pure tensors with the same argument going backwards. As for (a), we use the fact that $R_e,T_e$ are $*$-homomorphisms to get for all $a\in X_e,h\in\hilb{H},k\in\hilb{K}$:
		\[
		 \langle R_e(a)h,k \rangle  =  \langle h,R_e(a^*)k \rangle  =  \langle h,P_{\hilb{H}}R_e(a^*)k \rangle .
		\]
		But by \ref{inCoexLemma}, $P_{\hilb{H}}R_e(a^*)k = T_e(a^*)P_{\hilb{H}}k$, and we have:
		\[
		 \langle h,T_e(a^*)P_{\hilb{H}}k \rangle  =  \langle T_e(a)h,P_{\hilb{H}}k \rangle  =  \langle T_e(a)h,k \rangle 
		\]
		and we are done.
		\end{proof}
		The key observation we will need in this paper is Proposition 5.8 in \cite{Shalit2009}, which we repeat here:
		\begin{subsecProp} \label{repDilCpDil}
		Suppose $(\sps{Y},R,\hilb{K})$ is a dilation of $(\sps{X},T,\hilb{H})$. Then the CP-semigroup given by
		\[
		\Theta_s(a) = \tild{R}_s(I_{\sps{X}_s}\otimes a)\tild{R}_s^*
		\]
		is a dilation of the CP-semigroup given by
		\[
		\Psi_s(a) = \tild{T}_s(I_{\sps{X}_s}\otimes a)\tild{T}_s^*
		\]
		\end{subsecProp}
		This means that we are able to translate the problem of finding E-dilations of CP-semigroups, via $\Xi$, to a problem of finding isometric dilations of representations of subproduct systems. We also have the following proposition:
		\begin{subsecProp}
		If $\sps{X}$ is a subsystem of $\sps{Y}$ with projections $\setind{p}{s}{S}$, and $T$ a representation of $\sps{X}$, then $\set{T_s\circ p_s}_{s\in\sgp{S}}$ is a representation of $\sps{Y}$, and furthermore a dilation of $T$.
		\end{subsecProp}
		This follows immediately from the definitions. 
		This proposition explains our course of action in the next sections. Let $\sgp{S}(\gr{G})$ be the free semigroup with relations in an acyclic graph $\gr{G}$. In Section \ref{embedding section}, we will prove that any subproduct system indexed $\sgp{S}(\gr{G})$, embeds in a product system (actually, we will prove this only requiring $\gr{G}$ to be without 3-cycles, which is a stronger result). Then, in Section \ref{dilation section}, we will show that any representation of a product system indexed by $\sgp{S}$ has an isometric dilation. Using the two propositions above we will conclude that any CP-semigroup generated by a tuple of CP-maps commuting according to $\gr{G}$, has a simultaneous dilation.
		
		Lastly, we note that attempting to first embed our subproduct systems in product systems should not, in theory, make our job harder. This is due to the next fact:
		\begin{subsecRemark}
		If $\sps{X}$ has an isometric representation $T$, and $T_e$ is injective, then $\sps{X}$ must be a product system.
		\end{subsecRemark}
		The explanation for the remark is as follows: $T$ isometric with $T_e$ injective implies that $T$ is injective. Therefore, by Remark \ref{remark1} the E-semigroup defined by $\Theta_s(a) = \tild{T_s}(I_{\sps{X}_s}\otimes a)\tild{T_s}^*$ satisfies that $(\sps{X},T) \cong \Sigma(\Theta)$. But since this was an E-semigroup, $\sps{X}$ is a product system.
		
		Since our interests lie in dilating the identity representation of the Arveson-Stinespring subproduct system associated with a CP-semigroup, we will always start with a representation $T$ with $T_e$ injective (recall that in this case, $T_e$ was defined to actually be the identity). For any dilation $(\sps{Y},S,\sps{K})$ of $(\sps{X},T,\hilb{H})$, $S_e$ will remain injective, as by the definitions, $S_e(a)$ will always have $T_e(a)$ as a direct summand. If we assume further that $S$ is isometric, we must conclude that $S(x)^*S(x) = S_e(\langle x,x\rangle)$, which will mean that $S(x)=0$  if and only if $x=0$, thus $S$ must be injective. Therefore, in order that $(\sps{Y},S,\hilb{K})$ be an isometric dilation of an Arveson-Stinespring subproduct system of a CP-semigroup, $\sps{Y}$ must be a product system.
		
\section{Embedding $\sps{X}$ in a product system} \label{embedding section}
	\label{chapEmbed}
	Throughout this section, $\sgp{S} := \sgp{S}(\gr{G})$, where $\gr{G}$ is a graph which contains no 3-cycles, and $\vn{M}$ some von Neumann algebra. We denote by $\alp{A}$ the alphabet of $\sgp{S}$. $\sps{X}$ will be some subproduct system of $\vn{M}$-correspondences, indexed by $\sgp{S}$, with product maps $\set{U^\sps{X}_{s,t}}$. Our goal is to show that $\sps{X}$ embeds into a product system $\sps{Y}$, that is, that there is a morphism of subproduct systems $V:\sps{Y}\rightarrow \sps{X}$. The idea for this construction is mainly inspired by the constructions Solel used in \cite[Lemma 5.10]{Solel2006} to show the existence of a simultaneous dilation of two commuting CP-maps.
	
	We begin by constructing the fibers of $\sps{Y}$. We of course take $\sps{Y}_e:=\sps{X}_e = \vn{M}$. For any letter of the alphabet $s\in\alp{A}$, we define
	\[
	\sps{Y}_s:=\bigoplus_{t\in\alp{A}}\sps{X}_t.
	\]
	We also denote by $p_t$ the orthogonal projection of $\sps{Y}_t$ onto $\sps{X}_t$.
	For any word $s\in\sgp{S}$, we fix an arbitrary presentation $s=t_1\dots t_n$, and define
	\[
	\sps{Y}_s:= \sps{Y}_{t_1}\otimes\dots\otimes\sps{Y}_{t_n}.
	\]
	
	Before we proceed to define the product maps, we need to introduce the following notion:
	\begin{secDef}
	Let $\gr{G}$ be a graph,  with vertices $\alp{A}:=V(\gr{G})$.
	For a collection $\set{\corr{E}_s}_{s\in\alp{A}}$ of $\vn{M}$-correspondences, we call a collection $\set{f_{s,t}:s,t\in\alp{A},\set{s,t}\in E(\gr{G})}$ a \emph{unitary flip system} if $f_{s,t}:\corr{E}_s\otimes\corr{E}_t\rightarrow\corr{E}_t\otimes\corr{E}_s$ is a unitary bimodule map satisfying $f_{s,t}^*=f_{t,s}$ for all $s,t\in\alp{A}$. We will refer to the maps $f_{s,t}$ as \emph{flips}.
	\end{secDef}
	We wish now to define a unitary flip system on $\set{\sps{Y}_s}_{s\in\alp{A}}$, as a precursor to defining the product maps. We do this by first noting the following decompositions, for $s,t\in\alp{A}$:
	\begin{equation}
	\label{decomp}
	\begin{split}
	&\sps{Y}_s\otimes \sps{Y}_t = (\sps{X}_s\otimes \sps{X}_t)\oplus(\sps{X}_t\otimes \sps{X}_s)\oplus(\dots)\\
	&\sps{Y}_t\otimes \sps{Y}_s = (\sps{X}_t\otimes \sps{X}_s)\oplus(\sps{X}_s\otimes \sps{X}_t)\oplus(\dots),
	\end{split}
	\end{equation}
	where the dots are the same in both cases (recall that $\sps{Y}_s = \sps{Y}_t$). We can now define $f_{s,t}$, matricially, according to these decompositions:
	\[
	f_{s,t} = 
	\begin{pmatrix} 
		(U^\sps{X}_{t,s})^*U^\sps{X}_{s,t}	& 1-(U^\sps{X}_{t,s})^*U^\sps{X}_{t,s}	& 0\\ 
		1-(U^\sps{X}_{s,t})^*U^\sps{X}_{s,t}& (U^\sps{X}_{s,t})^*U^\sps{X}_{t,s} 	& 0\\
		0						& 0							& 1
	\end{pmatrix}
	\]
	Verifying that these are unitary bimodule maps is an easy calculation. So we see that $\set{f_{s,t}}_{s,t\in\alp{A}}$ constitute a unitary flip system.
	
	Given two presentations of a word $s=a_1\dots a_n = b_1\dots b_n$ we may define a map $f: \sps{Y}_{a_1}\otimes\dots\otimes\sps{Y}_{a_n}\rightarrow\sps{Y}_{b_1}\otimes\dots\otimes\sps{Y}_{b_n}$ by composing maps of the form $I\otimes f_{r,t}\otimes I$ in some order.
	\begin{secLemma}
	\label{combilemma}
	Let $\gr{G}$ be a graph without 3-cycles, $\sgp{S}$ the free semigroup with relations in $\gr{G}$, and $\alp{A}$ its alphabet. Then given any unitary flip system $\set{f_{s,t}:s,t\in\alp{A}}$ for $\set{\sps{Y}}_{s\in\alp{A}}$, the map $f$ defined as in the above paragraph is unique.
	\end{secLemma}
	The proof of the lemma is combinatorial in nature, and is quite long and tedious. We therefore move it to the appendix, so as to not break the flow.
	
	We are now ready to define the product maps in $\sps{Y}$. Recall that in the definition of the fibers of $\sps{Y}$, for each word in $\sgp{S}$ we have fixed some presentation.
	For $s,t\in\sgp{S}$, let us write the presentations we fixed as $s=a_1\dots a_n, t=b_1\dots b_k,st=c_1\dots c_{n+k}$. We then define $U_{s,t}^{\sps{Y}}$ to be the unique map $(\sps{Y}_{a_1}\otimes\dots\otimes\sps{Y}_{a_n})\otimes (\sps{Y}_{a_1}\otimes\dots\otimes\sps{Y}_{a_n})\rightarrow \sps{Y}_{c_1}\otimes\dots\otimes\sps{Y}_{c_{n+k}}$, as in Lemma \ref{combilemma}. If $s$ or $t$ were the empty word, we define $U_{s,t}^{\sps{Y}}$ to be the left or right action of $\sps{Y}_e$, accordingly. This definition obviously yields a unitary map, as the composition of unitaries. The associativity condition follows immediately from the uniqueness in Lemma \ref{combilemma}: $U_{rs,t}^{\sps{Y}}(U_{r,s}^{\sps{Y}}\otimes I_{\sps{Y}_t})$ and $U_{r,st}^{\sps{Y}}(I_{\sps{Y}_r}\otimes U_{s,t}^{\sps{Y}})$ are both maps as in the lemma, with the same range and domain, therefore they are equal.
	
	To show that $\sps{X}$ embeds in $\sps{Y}$, we proceed to define the morphism $V:\sps{Y}\rightarrow \sps{X}$. We obviously take $V_e$ to be the identity. For $s=a_1\dots a_n$, the fixed presentation we chose earlier, we define 
	\[
	V_s := U_{a_1,\dots,a_n}^{\sps{X}}\circ(p_{a_1}\otimes\dots\otimes p_{a_n}),
	\]
	where we recall that $U_{a_1,\dots,a_n}^{\sps{X}}$ is the unique map $\sps{X}_{a_1}\otimes\dots\otimes \sps{X}_{a_n}\rightarrow \sps{X}_s$ as in Remark \ref{uniqueremark}. This map is indeed co-isometric, as it is a composition of two co-isometries.
	We must now prove the identity:
	\begin{equation}
	\label{eq1}
	V_{st}\circ U^{\sps{Y}}_{s,t} = U^{\sps{X}}_{s,t}\circ (V_s\otimes V_t)
	\end{equation}
	For all $s,t\in\sgp{S}$. This will follow from the following proposition, which will clear up what is happening here.
	\begin{secProp}
	\label{diagramProp}
	For two presentations $s = a_1\dots a_k = b_1\dots b_k$ and the map $f:Y_{a_1}\otimes \dots \otimes Y_{a_k}\rightarrow Y_{b_1}\otimes \dots \otimes Y_{b_k}$ from Lemma \ref{combilemma}, we have that the following diagram commutes:
	\begin{align}
	\label{diagram1}
	\xymatrix{
	\sps{Y}_{a_1}\otimes \dots \otimes \sps{Y}_{a_k} \ar[d]^{p_{a_1}\otimes\dots\otimes p_{a_k}} \ar[rr]^{f} && \sps{Y}_{b_1}\otimes \dots \otimes \sps{Y}_{b_k} \ar[d]^{p_{b_1}\otimes\dots\otimes p_{b_k}}
	\\
	\sps{X}_{a_1}\otimes \dots \otimes \sps{X}_{a_k} \ar[dr]^{U^{\sps{X}}} && \sps{X}_{b_1}\otimes \dots \otimes \sps{X}_{b_k} \ar[dl]_{U^{\sps{X}}}
	\\
	&\sps{X}_s
	}
	\end{align}
	\end{secProp}
	\begin{proof}
	It is clear that it suffices to show this when the presentations differ only by one rewriting according to a commutation relation, and the general case will follow by induction. Thus the map 
	$\xymatrix{
	\sps{Y}_{a_1}\otimes \dots \otimes \sps{Y}_{a_k} \ar[rr]^{f}&& \sps{Y}_{b_1}\otimes \dots \otimes \sps{Y}_{b_k} 
	}$
	can be taken to be a flip, tensored with the identity. Following this reasoning, since the maps in question will act the same on all components that are not flipped, it really suffices to show that for commuting letters $r,t$, the following diagram commutes:
	\begin{align}
	\label{subdiagram1}
	\xymatrix{
	\sps{Y}_{r}\otimes \sps{Y}_{t} \ar[d]^{p_{r}\otimes p_{t}} \ar[rr]^{f_{r,t}} && \sps{Y}_{t}\otimes \sps{Y}_{r} \ar[d]^{p_{t}\otimes p_{r}}
	\\
	\sps{X}_{r}\otimes \sps{X}_{t} \ar[dr]^{U^{\sps{X}}_{r,t}} && \sps{X}_{t}\otimes \sps{X}_{r} \ar[dl]_{U^{\sps{X}}_{t,r}} 
	\\
	&X_s
	}
	\end{align}
	where $s=rt=tr$.
	This reduces to a calculation. It suffices to show commutativity of (\ref{subdiagram1}) when acting on pure tensors. Thus we take 
	$x=\begin{pmatrix}x_1\\x_2\\ \vdots\end{pmatrix}\in Y_r, z=\begin{pmatrix}z_1\\z_2\\ \vdots\end{pmatrix}\in Y_t$, where $x_1,z_1\in X_r$ and $x_2,z_2\in X_t$. On one hand, $(p_r\otimes p_t)(x\otimes z) = x_1\otimes z_2$, and this is mapped further to $U^X_{r,t}(x_1\otimes z_2)$.
	On the other hand, we recall (\ref{decomp}) and write, according to those decompositions, $x\otimes z=\begin{pmatrix}x_1\otimes z_2\\z_1\otimes x_2\\ *\end{pmatrix}$. we get:
	\begin{align*}
	f_{s,t}(x\otimes z) &= 
	\begin{pmatrix} 
	  (U^X_{t,r})^*U^X_{r,t}&1-(U^X_{t,r})^*U^X_{t,r}&0\\ 
	  1-(U^X_{r,t})^*U^X_{r,t}&(U^X_{r,t})^*U^X_{t,r}&0\\
	  0&0&1
	\end{pmatrix}
	\begin{pmatrix}x_1\otimes z_2\\z_1\otimes x_2\\ *\end{pmatrix}
	\\ &=\begin{pmatrix} (U^X_{t,r})^*U^X_{r,t}(x_1\otimes z_2)+ (1-(U^X_{t,r})^*U^X_{t,r})(z_1\otimes x_2)\\ *\\ *\end{pmatrix}
	\end{align*}
	The projection $p_t\otimes p_r$ now maps this to $(U^X_{t,r})^*U^X_{r,t}(x_1\otimes z_2)+ (1-(U^X_{t,r})^*U^X_{t,r})(z_1\otimes x_2)$. Applying $U^X_{t,r}$ we see:
	\begin{align*}
	U^X_{t,r}&((U^X_{t,r})^*U^X_{r,t}(x_1\otimes z_2)+ (1-(U^X_{t,r})^*U^X_{t,r})(z_1\otimes x_2))\\
	&=U^X_{t,r}(U^X_{t,r})^*U^X_{r,t}(x_1\otimes z_2)+ (U^X_{t,r}-U^X_{t,r}(U^X_{t,r})^*U^X_{t,r})(z_1\otimes x_2) \\ 
	&=U^X_{r,t}(x_1\otimes z_2) + (U^X_{t,r}-U^X_{t,r})(z_1\otimes x_2) \\
	&= U^X_{r,t}(x_1\otimes z_2)
	\end{align*}
	exactly as we wanted.
	\end{proof}
	
	We now notice that (\ref{eq1}) describes a diagram of the sort handled in Proposition \ref{diagramProp}: suppose the presentations we fixed are $s=a_1\dots a_k$, $t=b_1\dots b_l$, $st = c_1\dots c_{k+l}$, then (\ref{eq1}) corresponds to the diagram:
	\[
	\xymatrix{
	\sps{Y}_{a_1}\otimes \dots \otimes \sps{Y}_{a_k}\otimes \sps{Y}_{b_1}\otimes \dots \otimes \sps{Y}_{b_l} \ar[d]^{p_{a_1}\otimes\dots\otimes p_{a_{b_l}}} \ar[rrr]^{U^\sps{Y}} &&& \sps{Y}_{c_1}\otimes \dots \otimes \sps{Y}_{c_{k+l}} \ar[d]^{p_{c_1}\otimes\dots\otimes p_{c_{k+l}}}
	\\
	\sps{X}_{a_1}\otimes \dots \otimes \sps{X}_{a_k}\otimes \sps{X}_{b_1}\otimes \dots \otimes \sps{X}_{b_l} \ar[dr]^{U^\sps{X}} &&& \sps{X}_{c_1}\otimes \dots \otimes \sps{X}_{c_{k+l}} \ar[dll]_{U^\sps{X}}
	\\
	&X_{st}
	}
	\]
	which commutes according to the last proposition. To conclude, we rephrase what we have shown in the form of a theorem:
	\begin{secTheorem}
	\label{mainthm3}
	Let $\sps{X}$ be a subproduct system of $\vn{M}$-correspondences, indexed by $\sgp{S}(\gr{G})$, where $\gr{G}$ has no 3-cycles. Then there exists a product system $\sps{Y}$ (of $\vn{M}$-correspondences, indexed by $\sgp{S}(\gr{G})$) which $\sps{X}$ embeds in. That is, there is a morphism of subproduct systems $V:\sps{Y}\rightarrow\sps{X}$.
	\end{secTheorem}
\section{Isometric dilations of product system representations} \label{dilation section}
	Since in the last section we managed to embed the subproduct systems we are concerned with in product systems, we will restrict the discussion in this chapter to dilations of representations of product systems. We restrict the notion of dilation slightly: from now on we only consider a dilation of $(\sps{X},T,\hilb{H})$ to be of the form $(\sps{X},S,\hilb{K})$. That is, we don't enlarge the product system, only the space it is represented on. We will therefore usually just call $S$ a dilation of $T$, and not bother with the tuples $(\sps{X},T,\hilb{H})$,$(\sps{X},S,\hilb{K})$. This coincides with the definition in Solel's paper \cite{Solel2006}.
	\subsection{The isometric dilation property}
	
		To try to make the notation less cumbersome, we introduce the following definition:
		\begin{subsecDef}
		We say a unital semigroup $\sgp{S}$ has the \emph{isometric dilation property} (IDP) if any representation $T$ of any product system $\sps{X}$, indexed by $\sgp{S}$, has an isometric dilation.
		\end{subsecDef}
		So, with this notation, $\mathbb{N}$ and $\mathbb{N}^2$ have the IDP (For $\mathbb{N}$, this follows easily from Theorem and Definition 2.18 in \cite{Solel2006}, and for $\mathbb{N}^2$ this is shown in \cite[Theorem 4.4]{Solel2006}). However, $\mathbb{N}^3$ does not have it, which is a consequence of Parrott's example, as shown in \cite{Shalit2009}. Our goal in section is to prove that for an acyclic graph $\gr{G}$, $\sgp{S}(\gr{G})$ has the IDP.
		To do this, we begin laying some groundwork. In the following definition, $T$ and $S$ are representations of the same product system $\sps{X}$, on Hilbert spaces $\hilb{H},\hilb{K}$, respectively.
		\begin{subsecDef}
		$T$ and $S$ are said to be \emph{unitarily equivalent} if there is a unitary $U:\hilb{H}\rightarrow\hilb{K}$, such that for any $x\in\sps{X}$, $S(x) = UT(x)U^*$
		\end{subsecDef}
		We note that if $T$ is isometric and unitarily equivalent to $S$, then $S$ is isometric as well.
		\begin{subsecDef}
		We say that $T = T_1\oplus T_2$ if $T_1,T_2$ are representation of $\sps{X}$ on $\hilb{H_1},\hilb{H_2}$, such that $\hilb{H}=\hilb{H_1}\oplus\hilb{H_2}$, and for all $x\in\sps{X}$, $T(x) = T_1(x)\oplus T_2(x)$.
		\end{subsecDef}
		We remark that it is easy to see that $T$ decomposes as a direct sum $T_1\oplus T_2$ if and only if $\hilb{H}$ has a subspace $\hilb{H_1}$, which is reducing for all $T(x)$. Also, it is equally easy to see that if $T=T_1\oplus T_2$, then $T$ is a dilation of both $T_1$ and $T_2$.
		
	\subsection{Minimal dilations and a lifting lemma}
		\begin{subsecDef}
		Suppose $S$ is a dilation of $T$. Then it is called \emph{minimal} if 
		\[
		\hilb{K} = \bigvee_{x\in \sps{X}}S(x)\hilb{H}.
		\]
		\end{subsecDef}
		We proceed to prove a proposition, asserting the existence of a \emph{unique} minimal isometric dilation of any representation of a product system indexed by $\mathbb{N}$.
		\begin{subsecProp}
		\label{minimalProp}
		Let $\sps{X}=\setind{\sps{X}}{n}{\mathbb{N}}$ be a product system, with representation $T$ on a Hilbert space $\hilb{H}$. Then there is an isometric dilation $R$ of $T$, acting on a Hilbert space $\hilb{K}$, which is minimal. Furthermore, $R$ is unique, up to unitary equivalence.
		\end{subsecProp}
		This follows easily from the existence of a minimal isometric dilation of representations of $W^*$-correspondences, shown in \cite[Theorem and Definition 2.18]{Muhly2002}.
		\begin{subsecRemark}
		It is actually true that any isometric dilation of $T$ decomposes as a direct sum of isometric representations, with a copy of the minimal one being a summand. This can be verified  by noting, that if $S$ is a dilation of $T$, the subspace $\hilb{K}= \bigvee_{x\in \sps{X}}S(x)\hilb{H}$ will be reducing for all $S(x)$, and thus the restriction of $S$ to $\hilb{K}$ will be minimal. It is then easy to see that the remaining summand must also be isometric.
		So if we fix $R$ to be a minimal isometric dilation of $T$, we may assume without loss of generality that any isometric dilation $S$ of $T$ can be written as $S = R\oplus \alpha$ for some isometric representation $\alpha$.
		\end{subsecRemark}
		
		Next is a sort of lifting lemma for free semigroups with relations in a graph, which have the IDP. Given a product system indexed by a semigroup $\sgp{S}$, which has a subsemigroup $\sgp{R}$, we denote by $\restr{\sps{X}}{\sgp{R}}$ the product system $\setind{\sps{X}}{r}{R}$, with the product maps being the same as in $\sps{X}$. In the same way, for a representation $T$, we denote by $\restr{T}{\sgp{R}}$ its restriction to $\restr{\sps{X}}{\sgp{R}}$.
		\begin{subsecLemma}\label{liftingLemma}
		Let $\sgp{S}:=\sgp{S}(\gr{G})$ be the free semigroup with relations in $\gr{G}$, and assume further that $\sgp{S}$ has the IDP. Let $r$ be a letter in the alphabet of $\sgp{S}$ (that is, a vertex in $\gr{G}$), and let $\sgp{R}$ be the unital semigroup it generates (that is, $\sgp{R}$ is the free semigroup with relations in the graph of a single vertex, $r$, which is isomorphic to $\mathbb{N}$). Let $\sps{X}$ be a product system indexed by $\sgp{S}$, and $T$ a representation on a Hilbert space $\hilb{H}$. If $\restr{T}{\sgp{R}}$ has an isometric dilation $W$ on a Hilbert space $\hilb{K}\supseteq\hilb{H}$, then there is an isometric dilation $V$ of $T$, on a Hilbert space $L\supseteq\hilb{K}$, such that $\restr{V}{\sgp{R}}$ is a dilation of $W$.
		\end{subsecLemma}
		\begin{proof}
		According to Proposition \ref{minimalProp}, $W = W^0\oplus W^1$, where $W^0$ is the minimal isometric dilation of $\restr{T}{\sgp{R}}$, acting on a Hilbert space $\hilb{H_0}\supseteq\hilb{H}$ (and $W^1$ is isometric, acting on some Hilbert space $\hilb{H_1}$). Since $\sgp{S}$ has the IDP, $T$ has an isometric dilation $Q$. $\restr{Q}{\sgp{R}}$ also decomposes as a direct sum, and without loss of generality, it decomposes to $\restr{Q}{\sgp{R}} = W^2\oplus W^0$, where $W^2$ is isometric, acting on $\hilb{H_2}$. 
		We would like to define a representation of $\sps{X}$ by $Q\oplus W^1$, but to make sense of this we must first extend $W^1$ to all of $\sps{X}$. We do this in the most trivial way possible, by defining $W^1_s:=0$ for all $s\notin\sgp{R}$. We need to show that this retains the properties of a representation: the only property that might not hold is
		\[
		W^1_{st}(U_{s,t}(x\otimes y)) = W^1_s(x)W^1_t(y).
		\]
		If both $s$ and $t$ are in $\sgp{R}$, this obviously holds, as $W^1$ started out as a representation. If one of them is not in $\sgp{R}$, then the RHS is $0$. We must show that in this case, so is the LHS. But this is obvious, as $st$ cannot be in $\sgp{R}$ - any presentation of it will have letters other than $r$.
		
		Having extended $W^1$, we consider $Q\oplus W^1$. This is a representation of $\sps{X}$, and furthermore a dilation of $Q$, thus a dilation of $T$. Since $\sgp{S}$ has the IDP, it has an isometric dilation $V$, acting on a Hilbert space $\hilb{L}$. $\hilb{L}$ contains $\hilb{H_0}\oplus \hilb{H_1} = \hilb{K}$. $V$ is also an isometric dilation of $T$, by virtue of being an isometric dilation of $Q\oplus W^1$. When restricted to $\sgp{R}$, it is an isometric dilation of $\restr{(Q\oplus W^1)}{\sgp{R}}$, which is precisely $W^2\oplus W^0\oplus W^1$, and is thus an isometric dilation of $W^0\oplus W^1 = \restr{T}{\sgp{R}}$. Thus we are done.
		\end{proof}
	\subsection{IDP for free semigroups with relations in acyclic graphs}
		Some preparation is still in order.
		\begin{subsecLemma}\label{damnLemma}
		Let $\corr{E},\corr{F}$ be $W^*$-correspondences over a von Neumann algebra $\vn{M}$, with c.c. representations $(\sigma,T),(\sigma,S)$ on $\bdd{H}$, respectively. Then there exists a c.c. representation $(\sigma,V)$ of $\corr{E}\otimes\corr{F}$ which satisfies $V(e\otimes f) = T(e)S(f)$ for all $e\in\corr{E},f\in\corr{F}$.
		\end{subsecLemma}
		\begin{proof}
		We define:
		\[U:=\tild{T}\circ (I_\corr{E}\otimes \tild{S}).
		\]
		Using the bijection in Lemma \ref{bijection lemma}, we see that there is a representation $(\sigma,V)$ such that $\tild{V} = U$. It is now easy to check that this is the map that we wanted: for all $e\in\corr{E}, f\in\corr{F},h\in \hilb{H}$
		\[
		V(e\otimes f)h = \tild{V}(e\otimes f\otimes h) = \tild{T}\circ (I_\corr{E}\otimes \tild{S})(e\otimes f\otimes h) = T(e)S(f)h
		\]
		\end{proof}
		
		The next lemma is about bounded chains of co-extensions in $\bdd{\hilb{H}}$, which we will use in a limiting procedure in the proof of the main theorem of this section.
		
		\begin{subsecLemma}
		\label{chainLemma}
		Let $T_0,T_1,\dots$ be operators on Hilbert spaces $\hilb{H_0}\subseteq\hilb{H_1}\subseteq\dots$, respectively, which satisfy $\norm{T_n}\leq c$ for some constant $c$, for all $n\in\mathbb{N}$. Assume further that $T_{n+1}$ is a co-extension of $T_n$ for all $n\in\mathbb{N}$. Let $\hilb{H}:=\bigvee_{n=0}^{\infty}\hilb{H_n}$. Then there is a unique operator $T\in\bdd{H}$ which is a co-extension of $T_n$ for all $n\in\mathbb{N}$. Furthermore, $T$ is the SOT limit of the sequence $\setind{T_nP}{n}{\mathbb{N}}$, where $P_n$ is the orthogonal projection on $\hilb{H_n}$, and satisfies $\norm{T}\leq c$. If we have also that $T_n$ is an isometry for all $n$, then $T$ is an isometry as well.
		\end{subsecLemma}
		\begin{proof}
		We want to first define an operator $S$, extending $T_n^*$ for all $n$, and then take $T = S^*$. Since we want $T^*$ to extend $T_n^*$ for all $n$, $S$ must satisfy $Sh = T_n^*h$ for all $h\in\hilb{H_n}$. Because $\restr{T_{n+1}^*}{\hilb{H}_n} = T_n^*$, this definition yields a well defined operator on $\bigcup_{n=0}^{\infty}\hilb{H_n}$. This operator satisfies $\norm{Sh}\leq c\norm{h}$ for all $h\in \bigcup_{n=0}^{\infty}\hilb{H_n}$, because $\norm{T_n^*}\leq c$ for all $n$. Thus, $S$ extends uniquely to a bounded operator on $\hilb{H}$. We take $T := S^*$, and note that it is indeed a co-extension of $T_n$ for all $n$. It is also clear that it unique, and that $\norm{T} = \norm{S} \leq c$.
		
		To show that it is the SOT limit of the sequence $T_nP_n$ in $\bdd{H}$, we observe that the fact that $T$ is a co-extension of $T_n$ implies, in particular, that
		\[
		T_nP_n = P_nTP_n
		\]
		for all $n$. But the RHS converges in SOT to $T$, thus so does the LHS, as required.
		
		Lastly, if $T_n$ was an isometry for all $n$, then for any $h\in\bigcup_{n=0}^{\infty}\hilb{H_n}$, for $n$ large enough that $h\in\hilb{H}_n$, we will have $\norm{T_nP_nh} = \norm{h}$. Thus, $\norm{Th} = \norm{h}$, and it is easy to see that if an operator is isometric on a dense set, it is isometric.
		\end{proof}
		\begin{subsecCor}\label{chainCor}
		Let $\set{{T}^{n}}_{n\in\mathbb{N}}$ be a sequence of representations of a product system $\setind{\sps{X}}{s}{\sgp{S}}$, on Hilbert spaces $\hilb{H_1}\subseteq\hilb{H_2}\subseteq\dots$, such that $T^{n+1}$ is a dilation of $T^n$ for all $n$. Then there exists a limit representation $T$ of $\sps{X}$ on $\hilb{H}:=\bigvee_n H_n$, which satisfies:
		\begin{enumerate}
		\item 
		For any $x\in\sps{X}$, $T^n(x)P_n\overset{SOT}{\longrightarrow}T(x)$, where $P_n$ is the orthogonal projection on $\hilb{H_n}$.
		\item
		$T$ is a dilation of $T^n$ for all $n$.
		\item
		$\tild{T^n_s}Q_n\overset{SOT}{\longrightarrow}\tild{T}_s$ for all $s\in\sgp{S}$, where $Q_n$ is the orthogonal projection $\sps{X}_s\otimes_{\sigma}\hilb{H}\rightarrow\sps{X}_s\otimes_{\sigma}\hilb{H_n}$ 
		\end{enumerate}
		Furthermore, if $T^n$ are isometric for all $n$, then $T$ is isometric.
		\end{subsecCor}
		\begin{proof}
		If we fix $x\in\sps{X}$, $\set{T_n(x)}_{n\in\mathbb{N}}$ will satisfy the conditions of Lemma \ref{chainLemma}, and thus we will be able to define $T(x)$ to be the SOT limit of $T_n(x)P_n$, which will be a co-extension of $T_n(x)$, for all $n$. Therefore, according to Lemma \ref{coexLemma}, if we verify that this $T$ is indeed a representation of $\sps{X}$, we will be done. We start by verifying that for any $s,t\in\sgp{S}$ and $x\in\sps{X_s},y\in\sps{X_t}$,
		\begin{equation}\label{eq3}
		T_s(x)T_t(y) = T_{st}(U^\sps{X}(x\otimes y))
		\end{equation}
		holds.
		For any $n$, we have the similar equality $T^n_s(x)P_nT^n_t(y)P_n = T^n_{st}(U^\sps{X}(x\otimes y))P_n$, because $T^n$ is a representation. Both $T^n_s(x)$ and $T^n_t(y)$ are bounded sequences, so taking SOT limit on both sides of the equation yields (\ref{eq3}).
		
		It is easy to see that $T_e$ is in fact a direct sum of $*-$representations of $\vn{M}$, which are all normal, therefore $T_e$ itself is a normal $*-$representation. This is because for all $a\in\vn{M}$, $T^{n+1}_e(a)$ decomposes to a direct sum with $T^n_e(a)$ as a summand. Furthermore, $T_e$ is non-degenerate - for unital $C^*$-algebras (and in particular von Neumann algebras), representations are non-degenerate if and only if they are unital, and we see $T_e$ is unital because $T^n_e$ were all unital.
		It remains to show that for all $s\in\sgp{S}$, $(T_e,T_s)$ is a c.c. representation of $\sps{X}_s$. 
		
		We want to use Lemma 3.5 in \cite{Muhly1998}, which states that if $T$ is bounded and covariant (that is, respects the bimodule structure with respect to $T_e$), then it is completely contractive if and only if satisfies the matrix inequality
		\begin{equation}\label{eq4}
		(T_s(x_i)^*T_s(x_j))\leq (T_e(\langle x_i,x_j\rangle))
		\end{equation}
		We first note that $T_s$ is bounded. This is because for any $x\in\sps{X_s}$, from Lemma \ref{chainLemma}, since $\norm{T^n_s(x)}\leq \norm{x}$ for all $n$, we have also that $\norm{T_s(x)}\leq\norm{x}$, which means that $\norm{T_s}\leq 1$. Now, to see that $T$ respects the correspondence structure, we note that for all $a,b\in\vn{M}$ and $x\in\sps{X_s}$, we have that $T^n_s(axb) = T^n_e(a)T^n_s(x)T^n_e(b)$, and again we may take SOT limits and get $T_s(axb) = T_e(a)T_s(x)T_e(b)$. Lastly, we want to show (\ref{eq4}) for any choice $x_1,\dots,x_n\in\sps{X_s}$. We again see that we have this inequality when replacing $T$ with $T^n$ for all $n$. We denote by $A_n$ the matrix $(T^n_s(x_i)^*T^n_s(x_j)P_n)_{i,j}\in{M_n(\bdd{H})}$. Then it is easy to see $A_n$ converges in the SOT to the LHS of (\ref{eq4}). We have the inequality 
		\[
		A_n\leq (T^n_e(\langle x_i,x_j\rangle)P_n),
		\]
		but $T^n_e$ is a direct summand in $T_e$, so obviously 
		\[
		(T^n_e(\langle x_i,xj\rangle)P_n)\leq (T_e(\langle x_i,x_j\rangle)).
		\]
		So really we have $A_n\leq (T_e(\langle x_i,xj\rangle))$, and this kind of inequality survives SOT limits. 
		
		To prove (3), we note that it is clear that for any $s\in\sgp{S}$, $\tild{T}^{n+1}_s$ is a co-extension of $\tild{T}^n_s$ for all $n$, thus applying Lemma \ref{chainLemma} again, we get that they have a limit co-extension, and by uniqueness we conclude that it is $\tild{T}$.
		
		Lastly, we note that if $T^n$ is isometric for all $n$, then this simply means that $\tild{T}^n_s$ is isometric for all $s\in\sgp{S}$. But invoking the last part of Lemma \ref{chainLemma}, we see that this implies that $\tild{T}_s$ is isometric, thus $T$ is isometric.
		\end{proof}
		
		Next is a lemma to help us verify when representations are isometric.
		\begin{subsecLemma}\label{isometricLemma}
		Let $\sps{X}$ be a product system indexed by a unital semigroup $\sgp{S}$, $T$ a representation of $\sps{X}$ on $\hilb{H}$, and denote $\sigma:=T_e$. Let $\alp{A}$ be a set, such that $\alp{A}\cup\set{e}$ generates $\sgp{S}$. Then $T$ is isometric if and only if $(\sigma,T_a)$ is isometric for all $a\in\alp{A}$.
		\end{subsecLemma}
		\begin{proof}
		One implication is trivial. For the other, assume $(\sigma,T_a)$ is isometric for all $a\in\alp{A}$. It suffices to show that for all $a,b\in\alp{A}$, $(\sigma,T_{ab})$ is isometric, and proceed by induction. Denote by $U_{a,b}$ the product map $\sps{X}_a\otimes \sps{X}_b\rightarrow \sps{X}_{ab}$, then we notice that
		\begin{equation}\label{tabeq}
		\tild{T}_{ab} \circ (U_{a,b}\otimes I_{\hilb{H}}) = \tild{T}_a\circ (I_{\sps{X}_a}\otimes \tild{T}_b).
		\end{equation}
		As this is a composition of isometries, $\tild{T}_{ab}$ is an isometry, and consequently so is $T_{ab}$.
		\end{proof}
		
		We are now ready to show that the free unital semigroup on two generators has the IDP. We denote this semigroup by $\sgp{F}^+_2$, and its generators by $a,b$.
		\begin{subsecProp}\label{freeTwoProp}
		$\sgp{F}^+_2$ has the IDP.
		\end{subsecProp}
		\begin{proof}
		Let $\sps{X}$ be a product system indexed by $\sgp{F}^+_2$, with representation $T:\sps{X}\rightarrow\bdd{H_0}$. Our goal is to define a sequence of dilations that will approximate an isometric dilation, and take its limit, in a manner that will be made precise. It will be convenient to identify the fibers $\sps{X}_t$ with the tensor products of the fibers associated with the letters, and to take the product maps to be the identity maps. Namely, to assume given a word $t=t_1t_2\dots t_n$, that $\sps{X}_t = \sps{X}_{t_1}\otimes \sps{X}_{t_2}\otimes\dots\otimes \sps{X}_{t_n}$. This can be justified by noting that $X$ is indeed isomorphic to the product system with these fibers, with the morphism simply being the product maps $U^{\sps{X}}_{t_1,\dots,t_n}$.
		
		We first define $T^a:=\restr{T}{\langle a\rangle}$, which is a representation of $\restr{X}{\langle a\rangle}$ (where $\langle a\rangle$ is the unital semigroup generated by $a$, which is isomorphic to $\mathbb{N}$). Then $T^a$ has an isometric dilation $S^a$, acting on $\hilb{H_1}$, where $\hilb{H_0}\subseteq\hilb{H_1}$. We now define a representation of $\restr{\sps{X}}{\langle b\rangle}$ by $S^b:= \restr{T}{\langle b\rangle}\oplus 0$, that is $S^b(x) = T(x)\oplus 0\in\bdd{\hilb{H_0}\oplus (\hilb{H_1}\ominus \hilb{H_0})}$. We are now able to define a representation $S$ of $\sps{X}$ on $\hilb{H_1}$ in the obvious way: we define $S_e := S^a_e$, $S_a := S^a_a$ and $S_b := S^b_b$, and for a presentation $t=t_1\dots t_n$, we have, according to Lemma \ref{damnLemma} (using induction on the number of letters, and noting that $U^\sps{X}:\sps{X}_{t_1}\otimes\dots\otimes\sps{X}_{t_n}\rightarrow \sps{X}_t$ is a unitary bimodule map), a map $S_t: \sps{X}_t\rightarrow \bdd{H_1}$, such that $(S_e,S_t)$ is a c.c. representation and 
		\begin{equation}\label{eq5}
		S_t(x_1\otimes\dots\otimes x_n) = S_{t_1}(x_1)\dots S_{t_n}(x_n)
		\end{equation}
		for any choice of $x_i\in\sps{X}_{t_i}$. 
		Writing down explicitly what we got from Lemma \ref{damnLemma}, we see that 
		\begin{equation} \label{tilde equation}
		\tild{S}_t = \tild{S}_{t_1}\circ (I\otimes \tild{S}_{t_2}) \circ (I\otimes \tild{S}_{t_3})\circ\dots \circ (I\otimes \tild{S}_{t_n})
		\end{equation}
		
		We claim that $\set{S_t}_{t\in\sgp{F}^+_2}$ is a representation of $\sps{X}$, and a dilation of $T$. We need verify two things: To complete showing that it is a representation, we need to show that for all $t,r\in\sgp{F}^+_2$, $x\in\sps{X}_t$, $y\in\sps{X}_r$ we have $S_{tr}(x\otimes y) = S_t(x)S_r(y)$. To show that it is a dilation of $T$, we need to show that $\tild{S}_t$ is a co-extension of $\tild{T}_t$ for all $x\in\sps{X},t\in\sgp{F}^+_2$.
		
		We start with the first. We want to define a representation $R$ of $\sps{X}_t\otimes\sps{X}_r$, satisfying $R(x\otimes y) = S_t(x)S_r(y)$. We do this in a seemingly roundabout way: we take the map $\tild{S}_t\circ (I\otimes \tild{S}_r)$, and using the bijection in Lemma \ref{bijection lemma} yields $R$. Now, because of (\ref{eq5}), $R$ and $S_{tr}$ coincide on pure tensors (i.e, elements of the form $x_{t_1}\dotimes x_{t_n}\otimes x_{r_1}\dotimes x_{r_k}$). These are both bounded module maps, thus by Remark \ref{pure tensor remark}, we get that $S_{tr}=R$, which in particular implies $S_{tr}(x\otimes y) = S_t(x)S_r(y)$ for all $x,y$.
		
		For the latter, we recall that for any letter $r\in \{a,b\}$, we have that $\restr{\tild{S}_r^*}{\hilb{H_0}} = \tild{T}^*_r$. Thus it is easy to see that Equation \ref{tilde equation} implies that for a general word $t=t_1\dots t_n$, $\restr{\tild{S}^*_t}{\hilb{H_0}} = \tild{T}^*_t$.
		
		Finally, we may proceed in building our sequence of dilations. We define $V^0=T$, $V^1:=S$, and continue this construction inductively: We build a new $S$, reversing the roles of $a$ and $b$: This time we take $S^b$ to be an isometric dilation of $\restr{V^1}{\langle b \rangle}$ and $S^a$ to be $V^1\oplus 0$. We then define $S$ as before, and define $V^2:=S$, acting on $\hilb{H_2}$. We get that $V^2$ is a dilation of $V^1$. We do this inductively, at each step reversing the roles of $a$ and $b$, and get a sequence $V^1,V^2,\dots$ of representations, each one a dilation of its predecessors. At the odd steps, $\restr{V^n}{\langle a \rangle}$ is isometric, and at the even steps $\restr{V^n}{\langle b \rangle}$ is isometric. We get, according to Corollary \ref{chainCor}, a limit representation, which we will call $V$, which is a dilation of $V^n$ for all $n$, and in particular of $T$. It remains only to show that $V$ is isometric. We denote for simplicity $\sigma_n = V^n_e$ and $\sigma = V_e$.
		
		We see that according to Lemma \ref{isometricLemma}, it suffices to show that $(\sigma,V_a)$ and $(\sigma,V_b)$ are isometric. As we have noted, $(\sigma_{2n}, V^{2n}_b)$ is isometric for all $n$. By Corollary \ref{chainCor}, we have that $\tild{V_b}$ is the SOT limit of $\tild{V}^{2n}_bQ_n$, thus it is easy to conclude that $\tild{V_b}$ is isometric (this is in fact, the last part of Lemma \ref{chainLemma}).
		
		For $a$, we do the same thing, but with the subsequence $V^{2n+1}_a$.
		\end{proof}
		
		We now arrive at the main theorem of this section. We will show that if $\gr{G}$ is an acyclic graph, and $\sgp{S}$ the free semigroup with relations in it, then $\sgp{S}$ has the IDP. We will use inductively the facts that $\mathbb{N}^2$ and the free unital semigroup on two generators have the IDP. The methods we will use in building the dilation are very similar to the ones used in the previous proposition, so we will omit some of the details.
		\begin{subsecTheorem}\label{mainthm4}
		Let $\gr{G}$ be an acyclic graph, and $\sgp{S}$ the free semigroup with relations in $\gr{G}$. Then $\sgp{S}$ has the IDP.
		\end{subsecTheorem}
		\begin{proof}
		The proof is by induction on $N$, the number of vertices in $\gr{G}$. For $N=1$, $\sgp{S}=\mathbb{N}$, and this will be the basis of the induction. We note that for $N=2$, $\sgp{S}$ is either $\mathbb{N}^2$ or $\sgp{F}^+_2$. As stated in the beginning of the section, the cases of $\mathbb{N},\mathbb{N}^2$ are known, and the case of $\sgp{F}^+_2$ is exactly the content of Proposition \ref{freeTwoProp}.
		
		So, let $\gr{G}$ be an acyclic graph with $N$ vertices, and let $\sgp{S}$ be the free semigroup with relations in $\gr{G}$. Then we can always find in $\gr{G}$ a vertex which has degree less or equal to one: $\gr{G}$ is a forest, so we can choose a leaf of a tree, or, if $\gr{G}$ has no edges at all, choose a vertex with degree zero. Let us distinguish this vertex and call it $a$, and if it was a leaf, let $b$ be its only neighbor. Otherwise, let $b$ be some arbitrary different vertex. In this way, if we consider $\gr{G}_0$ to be the induced subgraph on the vertices $a,b$, the free semigroup with relations in $\gr{G}_0$ is either $\mathbb{N}^2$ or $\sgp{F}^+_2$, and at any rate has the IDP. Let us call it $\sgp{S}_0$, and we note that we can view $\sgp{S}_0$ as a subsemigroup of $\sgp{S}$. Now, let us denote by $\gr{G}_1$ the subgraph induced on $V(\gr{G})\setminus \set{a}$, and $\sgp{S}_1$ the free semigroup with relations in $\gr{G}_1$. Then by the induction hypothesis, $\sgp{S}_1$ has the IDP, and it too can be viewed as a subsemigroup of $\sgp{S}$.
		
		Let $\sps{X}$ be a product system of $\vn{M}$ correspondences, indexed by $\sgp{S}$, and $T$ a representation of $\sps{X}$ on a Hilbert space $\hilb{H}$. We wish to find an isometric dilation of $T$. We will build inductively two sequences of representations, $V^n$ of $\restr{T}{\sgp{S}_0}$ and $W^n$ of $\restr{T}{\sgp{S}_1}$. We begin by simply taking $V^0=\restr{T}{\sgp{S}_0}$, and taking $W^0$ to be some isometric dilation of $\restr{T}{\sgp{S}_1}$. By Lemma \ref{liftingLemma}, we can take $V^1$ to be an isometric dilation of $V^0$, such that $\restr{V^1}{\langle b\rangle}$ is a dilation of $\restr{W^0}{\langle b\rangle}$. In the same way, we can take $W^1$ to be an isometric dilation of $W^0$, such that $\restr{W^1}{\langle b\rangle}$ is a dilation of $\restr{V^1}{\langle b\rangle}$. We continue this process inductively, that is, given $V^n,W^n$, we take $V^{n+1}$ to be an isometric dilation of $V^n$, with $\restr{V^{n+1}}{\langle b\rangle}$ being a dilation of $\restr{W^n}{\langle b\rangle}$, and take $W^{n+1}$ to be an isometric dilation of $W^n$, with $\restr{W^{n+1}}{\langle b\rangle}$ being a dilation of $\restr{V^{n+1}}{\langle b\rangle}$. We get, by Corollary \ref{chainCor}, limit representations $V,W$ on a Hilbert space $\hilb{K}\supseteq \hilb{H}$. 
		
		We first note that $V,W$ are isometric: This is because they are limits of isometric representations, and thus follows from Corollary \ref{chainCor}. 
		
		Next, we note that $V$ and $W$ coincide on $\restr{\sps{X}}{\langle b \rangle}$: for $s\in\langle b \rangle$ and $x\in\sps{X}_s$, $V_s(x)$ is the SOT limit of $V^n_s(x)$, and $W_s(x)$ the SOT limit of $W^n_s(x)$. But for any $n$, $W^n_s(x)$ was, from our construction, a dilation of $V^n_s(x)$, and $V^{n+1}_s$ a dilation of $W^n_s(x)$. Thus, one can form the sequence $V^0_s(x),W^0_s(x),V^1_s(x),W^1_s(x),\dots$, which according to Lemma \ref{chainLemma} will converge in the SOT. But then both $V_s(x)$ and $W_s(x)$ are partial limits of this sequence, and in particular are equal.
		
		Seeing that $V$ and $W$ coincide on the intersection of their domains (namely, $\langle b \rangle$), it makes sense to drop the distinction between them, and just refer to both of them as $V$. So we want to extend $V$ to a representation of all of $\sps{X}$. Since $V$ is already defined on $\sps{X}_s$ for any $s$ in the alphabet of $\sgp{S}$ (as it is defined on all $s\in\sgp{S}_0\cup\sgp{S}_1$), there is really only one way to do this. Like we did in the proof of Proposition \ref{freeTwoProp}, for each $s\in\sgp{S}$, we pick a presentation $s=s_1\dots s_n$, and let $V_s$ be the map satisfying
		\[
		V_s(U^\sps{X}(x_1\otimes\dots\otimes x_n)) = V(x_1)\dots V(x_n),
		\]
		such that $(V_e,V_s)$ is a c.c. representation of $\sps{X}_s$. This maps exists due to inductive application of Lemma \ref{damnLemma}. Specifically we choose 
		\begin{equation} \label{tilde equation 2}
		\tild{V}_s = \tild{V}_{s_1}\circ (I\otimes \tild{V}_{s_2}) \circ\dots \circ (I\otimes \tild{V}_{s_n})\circ ((U^X_{s_1,\dots ,s_n})^*\otimes I_{\hilb{K}})
		\end{equation}
		and take the $V_s$ corresponding to $\tild{V}_s$.
		We note that this agrees with the previous definition of $V$ wherever it was already defined. Checking that this is a representation is again similar to what we did in Proposition \ref{freeTwoProp}, with one complication: to see that
		\begin{equation}\label{eq7}
		V_{st}(U^\sps{X}(x\otimes y)) = V_s(x)V_t(y)
		\end{equation}
		holds for all appropriate $x,y\in\sps{X}$, we need to work a little bit harder. However, if we do manage to show this, we are done, as $V$ is isometric. This is because it is isometric on $\sgp{S}_0\cup\sgp{S}_1$, which with $e$ generate all of $\sgp{S}$, thus we may apply Lemma \ref{isometricLemma}.
		
		Returning to (\ref{eq7}), we see that it will follow if we show that, when we defined $V_s$ by (\ref{tilde equation 2}) according to a presentation $s=s_1\dots s_n$, we could have chosen any other presentation and gotten the same map. That is, that the definition of $V_s$ was independent of the choice of presentation of $s$. Let us first see that this is indeed the case.
		
		We will start by showing this for two lettered words. Let $s$ be a word with presentations $s=s_1s_2 = s_2s_1$. Then either $s_1,s_2\in\sgp{S}_0$ or $s_1,s_2\in\sgp{S}_1$. That is because as vertices in $\gr{G}$, $s_1,s_2$ must be connected by and edge (or be the same vertex). If at least one of them is $a$, the other must be either $b$ or $a$, because we took $a$ to be either adjacent only to $b$, or with degree zero. So if one of them is $a$, they are both in $\gr{G}_0$, and if none of them is $a$, they are both in $\gr{G}_1$, and in any case, they are in the same sub semigroup ($\sgp{S}_0$ or $\sgp{S}_1$). But in that case, it is obvious that the presentations did not matter in the definition of $V_s$, as $\restr{V}{\sgp{S}_0}$ and $\restr{V}{\sgp{S}_1}$ were already representations, and our new definition of $V_s$ coincided with the old one. Explicitly, we can write $\tild{V}_{s_1}\circ(I_{\sps{X}_{s_1}}\otimes \tild{V}_{s_2})\circ((U^{\sps{X}}_{s_1,s_2})^*\otimes I_{\hilb{K}}) = \tild{V}_{s_2}\circ(I_{\sps{X}_{s_2}}\otimes \tild{V}_{s_1})\circ((U^{\sps{X}}_{s_2,s_1})^*\otimes I_{\hilb{K}})$, thus denoting $f=(U^{\sps{X}}_{s_2,s_1})^*U^{\sps{X}}_{s_1,s_2}$ the flip from $\sps{X}_{s_1}\otimes \sps{X}_{s_2}$ to $\sps{X}_{s_2}\otimes \sps{X}_{s_1}$, we have
		\begin{equation} \label{twolettereq}
		\tild{V}_{s_1}\circ(I_{\sps{X}_{s_1}}\otimes \tild{V}_{s_2}) = \tild{V}_{s_2}\circ(I_{\sps{X}_{s_2}}\otimes \tild{V}_{s_1})\circ(f\otimes I_{\hilb{K}})
		\end{equation}
		
		Now let $s$ be a general word. If we show that that given two presentations of $s$, which differ by just one application of a commutation relation, the definition of $V_s$ for them coincides, then the rest follows easily by induction. So this is what we will show. Without loss of generality, we assume that they differ on the first two letters, that is, the two presentations are $s=s_1s_2\dots s_n = s_2s_1\dots s_n$ (so $s_1s_2 = s_2s_1$ is the commutation relation applied). To avoid confusion, we denote by $V_s$ the maps defined by the presentation $s=s_1s_2\dots s_n$, and by $V'_s$ the map define by the presentation $s=s_2s_1\dots s_n$. So
		\[
		\tild{V'}_s = \tild{V}_{s_2}\circ(I\otimes\tild{V}_{s_1})\circ(I\otimes\tild{V}_{s_3})\circ\dots\circ (I\otimes\tild{V}_{s_n})\circ ((U^{\sps{X}}_{s_2,s_1,s_3\dots ,s_n})^*\otimes I_{\hilb{K}})
		\]
		and
		\begin{align*}
		\tild{V}_s&=\tild{V}_{s_1}\circ(I\otimes\tild{V}_{s_2})\circ(I\otimes\tild{V}_{s_3})\circ\dots\circ (I\otimes\tild{V}_{s_n})\circ ((U^{\sps{X}}_{s_1,\dots ,s_n})^*\otimes I_{\hilb{K}}) = \\
		&=\tild{V}_{s_2}\circ(I\otimes\tild{V}_{s_1})\circ(f\otimes I_{\hilb{K}})\circ(I\otimes\tild{V}_{s_3})\circ\dots\circ (I\otimes\tild{V}_{s_n})\circ ((U^{\sps{X}}_{s_1,\dots ,s_n})^*\otimes I_{\hilb{K}}).
		\end{align*}
		But it is easy to see that in the last equation, the expression $f\otimes I_{\hilb{K}}$ can move to the right, and since $f\circ (U^{\sps{X}}_{s_1,\dots ,s_n})^* = (U^{\sps{X}}_{s_2,s_1,s_3\dots ,s_n})^*$ we will get that $\tild{V}_s=\tild{V'}_s$. More rigorously, one can see on pure tensors, $(f\otimes I_{\hilb{K}})\circ(I\otimes\tild{V}_{s_3})\circ\dots\circ (I\otimes\tild{V}_{s_n})\circ ((U^{\sps{X}}_{s_1,\dots ,s_n})^*\otimes I_{\hilb{K}})$ acts the same as $(I\otimes\tild{V}_{s_3})\circ\dots\circ (I\otimes\tild{V}_{s_n})\circ (f\circ(U^{\sps{X}}_{s_1,\dots ,s_n})^*\otimes I_{\hilb{K}})$, and from there proceed as in the proof of Proposition \ref{freeTwoProp}.
		
		We have thus seen that $V_s$ is defined the same way for all presentations of $s$. We still need to explain why this is enough to show that (\ref{eq7}) holds. Let $s=s_1\dots s_n$, $t=t_1\dots t_m$ be presentations of $s,t$. Then $st = s_1\dots s_nt=t_1\dots t_m$ is a presentation of $st$. Thus for pure tensors we would have
		
		\begin{align*}
		&V_{st}(U^\sps{X}(U^\sps{X}(x_1\otimes\dots\otimes x_n)\otimes U^\sps{X}(y_1\otimes\dots\otimes y_m))) \\
		&= V_{st}(U^\sps{X}(x_1\otimes\dots\otimes x_n\otimes y_1\otimes\dots\otimes y_m))\\
		&= V(x_1)\dots V(x_n) V(y_1)\dots V(y_m).
		\end{align*}
		On the other hand, it is also clear that
		\begin{align*}
		&V_s(U^\sps{X}(x_1\otimes\dots\otimes x_n))V_t(U^\sps{X}(y_1\otimes\dots\otimes y_m))\\ 
		&=V(x_1)\dots V(x_n) V(y_1)\dots V(y_m),
		\end{align*}
		so we would have (\ref{eq7}) for pure tensors $x,y$ (and thus for $x,y$ in the linear span of the pure tensors). It now follows as in the proof of Proposition \ref{freeTwoProp} that\comm{
			We now proceed in a manner similar to what we did in the last proposition. We consider the maps \[\tild{V}_{st}\circ U^{\sps{X}}\] and \[\tild{V}_{s}\circ(I\otimes \tild{V}_{t})\circ ((U^{\sps{X}}_{s_1,\dots,s_n})^*\otimes (U^{\sps{X}}_{t_1,\dots,t_m})^*\otimes I_{\hilb{K}}).\] By the bijection in Lemma \ref{bijection lemma}, they are given by $\tild{A}$ and $\tild{B}$ for some representation $A,B$ of $X_{s_1}\otimes\dots\otimes X_{s_n}\otimes X_{t_1}\otimes\dots\otimes X_{t_m}$. The calculation above implies that $A$ and $B$ coincide on pure tensors, therefore by Remark \ref{pure tensor remark}, $A=B$. So we can conclude that 
		}
		(\ref{eq7}) holds for all $x,y$.

		To conclude, we see that $V$ is a well defined isometric representation of $\sps{X}$, and the only question that may remain is whether this is a dilation of $T$ or not. The answer is, of course, in the affirmative: we have that $\tild{V}_s$ is a co-extension of $\tild{T}_s$, whenever $s$ is in the alphabet (or is $e$), by virtue of $\restr{V}{\sgp{S}_0}$ and $\restr{V}{\sgp{S}_1}$ being dilations of $\restr{T}{\sgp{S}_0}$ and $\restr{T}{\sgp{S}_1}$. Thus we can use the same argument as in Proposition \ref{freeTwoProp} to see that for any $s\in\sgp{S}$, $\tild{V}_s$ is a co-extension of $\tild{T}_s$ (albeit writing down explicitly what $\tild{V}_s$ is in terms of the presentation of $s$ might be a little more cumbersome here, which is the main reason we omit it). This completes the proof.
		\end{proof}
		~
	
\section{Dilations of CP-maps commuting according a graph}~	\label{cp maps section}

	We wish to combine our results in the previous chapters in the form of a theorem on simultaneous dilations of CP-maps, commuting according to an acyclic graph.
	\begin{secTheorem}
	Let $\gr{G}$ be an acyclic graph, and $\sgp{S}$ the free semigroup with relations in $\gr{G}$. Let $\varphi = \setind{\varphi}{s}{S}$ be a CP-semigroup acting on $\vn{M}\subseteq \bdd{H}$. That is, for every element of the alphabet $a\in\alp{A}$, $\varphi_a$ is a CP-map, and $\setind{\varphi}{a}{\alp{A}}$ commute according to $\gr{G}$. Then $\varphi$ has an E-dilation.
	\end{secTheorem}
	\begin{proof}
	We have actually already outlined the proof near the end of the preliminaries. We define $(\sps{X},T)$ to be the Arveson-Stinespring subproduct system (and $T$ the identity representation representation) associated to $\varphi$. So $\sps{X}$ is a subproduct system of $\vn{M}'$-correspondences, and $T$ acts on $\bdd{\hilb{H}}$. Then $\varphi$ is given by 
	\[
	\varphi_s(a) = \tild{T_s}\circ (I_{\sps{X}_s}\otimes a)\circ \tild{T_s}^*.
	\]
	In \cite{Shalit2009} (Proposition 5.8), Shalit and Solel show that given a dilation $(Y,R,K)$ of $(X,T,H)$, the CP-semigroup $\alpha$ given by 
	\[
	\alpha_s(a)= \tild{R}_s\circ (I_{\sps{Y}_s}\otimes a)\circ \tild{R_s}^*,
	\]
	acting on $R_e(\vn{M}')'\subseteq \bdd{\hilb{K}}$, will be a dilation of $\varphi$. If $R$ is isometric, then $\alpha_s$ will be a $*$-endomorphism for all $s$, thus an $\alpha$ will be an E-dilation of $\varphi$. So we see that it suffices to show that $(X,T,H)$ has an isometric dilation.
	
	We have shown in Theorem \ref{mainthm3} that $X$ embeds in a product system, so we may assume $\sps{X}$ is a subsystem of a product system $\sps{Y}$, where $V:\sps{Y}\rightarrow\sps{X}$ is the morphism showing this, given simply by the projections of $\sps{Y}_s$ onto $\sps{X}_s$ for all $s\in\sgp{S}$. We note that $S:=T\circ V$ is a representation of $\sps{Y}$, and furthermore, $(\sps{Y},S,\hilb{H})$ is a dilation of $(X,T,\hilb{H})$ - this is an easy consequence of the definitions of morphisms and dilations. We now note that by Theorem \ref{mainthm4}, $\sgp{S}$ has the isometric dilation property, thus there is an isometric dilation $(\sps{Y},R,\hilb{K})$ of $(\sps{Y},S,\hilb{H})$, therefore it is an isometric dilation of $(\sps{X},T,\hilb{H})$, and we are done.
	\end{proof}
	
	In \cite[Theorem 5.14]{Shalit2009} Shalit and Solel show an example of three pairwise commuting CP-maps, with no simultaneous pairwise commuting E-dilations. That is, a CP-semigroup $\varphi = \set{\varphi_n}_{n\in\mathbb{N}^3}$ which has no E-dilation. Its existence essentially follows from Parrott's example of three commuting contractions, which have no isometric (commuting) dilation. 
	
	We wish to show that such an example also exists when we consider tuples of CP-maps commuting according to a graph $\gr{G}$ with cycles of any length. We will follow the proof in \cite{Shalit2009}, but replace the role of Parrott's example with a tuple of contractions commuting according to $\gr{G}$, which does not have an isometric co-extension, commuting according to $\gr{G}$. 
	
	In \cite{Opvela2006} (Theorem 2.3), Opela gives an example which almost works for our purposes - his example is of a tuple of contractions which does not have a unitary dilation, instead of not having an isometric co-extension, as we need. To overcome this technicality, we prove the next lemma and get the example we need as a corollary. The lemma is very similar to Lemma \ref{chainLemma}.
	\begin{secLemma}\label{chainLemma2}
	Let $\setind{T}{n}{\mathbb{N}}$ be operators on Hilbert spaces $\hilb{H_0}\subseteq\hilb{H_1}\subseteq\dots$, respectively, which satisfy $\norm{T_n}\leq c$ for some constant $c$, for all $n\in\mathbb{N}$. Assume further that $T_{n+1}$ is a dilation of $T_n$ for all $n\in\mathbb{N}$. Let $\hilb{H}:=\bigvee_{n=0}^{\infty}\hilb{H_n}$. Then the sequence $\setind{T_nP}{n}{\mathbb{N}}$ converges in the SOT to an operator $T\in\bdd{H}$, where $P_n$ is the orthogonal projection on $\hilb{H_n}$. Furthermore, $T$ is a dilation of $T_n$, for all $n\in\mathbb{N}$, and $\norm{T}\leq c$. If $T_n$ is an isometry for all $n$, then so is $T$.
	\end{secLemma}
	\begin{proof}
	We begin by showing that for any $x\in\bigcup_{n=0}^{\infty}\hilb{H_n}$, the sequence $T_nP_nx$ converges. We restrict ourselves to values of $n$ large enough such that $x\in\hilb{H_n}$. Thus $T_nP_nx = T_nx$. We notice that
	\[
	T_{n+1}x = P_nT_{n+1}x + Q_nT_{n+1}x,
	\]
	where $Q_n$ is the orthogonal projection on $\hilb{H_{n+1}}\ominus\hilb{H_n}$. But $T_{n+1}$ is a dilation of $T_n$, so we have $P_nT_{n+1}x=T_nx$. We conclude that for large enough values of $n\leq m$, we have $(T_m-T_n)x = \sum_{k=n}^{m-1}Q_kT_{k+1}x$, thus:
	\[
	\norm{(T_m-T_n)x}^2 = \sum_{k=n}^{m-1}\norm{Q_kT_{k+1}x}^2.
	\]
	But the LHS is bounded by $2c^2\norm{x}^2$, so we can conclude that $\sum_{k=n}^{\infty}\norm{Q_kT_{k+1}x}^2$ converges, and in particular $\sum_{k=1}^{\infty}\norm{Q_kT_{k+1}x}^2$ converges. 
	This implies that $\norm{(T_m-T_n)x}^2$ are partial sums of a convergent positive series, therefore $T_nx$ is Cauchy, and has a limit in $\hilb{H}$. We define $Tx$ to be this limit.
	$T$ is obviously bounded, as $\norm{Tx} = \lim_n\norm{T_nP_nx}\leq c\norm{x}$, therefore it extends uniquely to an element of $\bdd{H}$, with norm less or equal to $c$. We want to show that $T$ is the SOT limit of $T_nP_n$, that is, that $Tx$ is the limit of $T_nP_nx$ for any $x\in\hilb{H}$. This fairly straightforward: we start by picking $x_0\in\bigcup_{n=0}^{\infty}$ such that $\norm{x-x_0}<\frac{\epsilon}{3c}$, and take $N$ large enough so for any $n>N$, $\norm{T_nP_nx_0-Tx_0}<\frac{\epsilon}{3}$. Then:
	\begin{align*}
	\norm{T_nP_nx_0-Tx}&\leq \norm{T_nP_nx-T_nP_nx_0}+\norm{T_nP_nx_0-Tx_0}+\norm{Tx_0-Tx}\\
	&\leq c\norm{x-x_0}+\frac{\epsilon}{3}+c\norm{x-x_0}<\epsilon.
	\end{align*}
	
	We must show that $T$ is a dilation of $T_n$, for all $n$. Specifically:
	\begin{equation}
	\begin{split}
	&P_n\restr{T^k}{\hilb{H_n}} = T^k_n
	\end{split}
	\end{equation}
	for all $k$.
	This is immediate: For any $k$, $T^k$ is the SOT limit of $(T_nP_n)^k$ (the sequences here are bounded). It is also clear that $(T_nP_n)^k = T_n^kP_n$. Let $h\in\hilb{H_n}$.
	For any $m>n$ we will have $P_mh = h$, thus 
	\[
	P_nT^nh = P_n\lim_{m\rightarrow\infty}T^k_mP_mh =  \lim_{m\rightarrow\infty}P_nT^k_mh.
	\]
	But again, if $m>n$ then the $P_nT^k_mh = T^k_nh$, since $T_m$ is a dilation of $T_n$, and thus we get $P_nT^nh = T^k_nh$, and we are done.
	
	Lastly, if $T_n$ is an isometry for all $n$, it will follow that $T$ is an isometry. Pick $h\in \hilb{H}_n$ for some $n$, then $T_mP_mh\rightarrow Th$, and thus $\norm{T_mP_mh}\rightarrow \norm{Th}$. But for $m>n$, we have that $\norm{T_mP_mh} = \norm{T_mh} = \norm{h}$, so we see that $T$ is isometric on a dense subspace of $\hilb{H}$, and is thus an isometry.
	\end{proof}
	\begin{secCor}\label{example}
	Let $\gr{G}$ be a graph containing a cycle, with $\sgp{S}$ the free semigroup with relations in $\gr{G}$. Then there is a semigroup of contractions $\setind{T}{s}{\sgp{S}}$ on a Hilbert space $\hilb{H}$, such that there is no semigroup $\setind{V}{s}{\sgp{S}}$ of isometries on a Hilbert space $\hilb{K}$, where $V_s$ is a co-extension of $T_s$ for all $s$.
	\end{secCor}
	\begin{proof}
	The example Opela gives in \cite{Opvela2006} yields a semigroup $R_0 =\set{T_{0,s}}_{s\in\sgp{S}}$, such that there is no semigroup of unitaries dilating it. Suppose for contradiction that the corollary is false. Then we may construct a semigroup of isometries $R_1 =\set{T_{1,s}}_{s\in\sgp{S}}$, co-extending $R_0$. But if the corollary is false, it is also true we may construct a semigroup $R_2 =\set{T_{2,s}}_{s\in\sgp{S}}$ of co-isometries, extending $R_1$. That is, $T_{2,s}$ is a co-isometry for all $s$, and is an extension of $T_{1,s}$. Truly, if this was not the case, then $\set{T_{1,a}^*}_{a\in\alp{A}}$ would generate a semigroup which would be the example we are looking for. We may thus continue this process inductively, taking isometric co-extensions in the odd steps, and co-isometric extensions in the even steps.
	
	This process yields, for all $s$, a sequence $T_{n,s}$ of dilations, which by Lemma \ref{chainLemma2} admits a limit operator $T_s$. It is easy to see that $R =\set{T_s}_{s\in\sgp{S}}$ will be a semigroup. It is also clear, that since there is a subsequence of isometries in $T_{n,s}$, $T_s$ must be an isometry. But similarly, $T^*_{n,s}$ is a sequence of dilations (with a subsequence of isometries), which will admit an isometric limit $K_s$. Verifying that $K_s = T^*_s$ is easy: for any $n$ we have that $P_nT_sP_n = P_nT_{n,s}P_n$ (since $T_s$ dilates $T_{n,s}$), and thus $P_nT^*_sP_n = P_nT^*_{n,s}P_n$. But we also have $P_nK_sP_n = P_nT^*_{n,s}P_n$, and thus 
	\[
	P_nK_sP_n = P_nT^*_sP_n.
	\]
	Taking SOT limits on both sides as $n$ tends to infinity, we get the equality we wanted.
	
	To conclude, we see that both $T_s$ and $T^*_s$ are isometries, and thus the semigroup $R$ is a semigroup of unitaries dilating $R_0$, which is a contradiction, and we are done.
	\end{proof}
	
	With this technicality taken care of, we may proceed to show the result we were seeking:.
	
	\begin{secTheorem}
	Let $\gr{G}$ be a graph containing a cycle, and $\sgp{S}$ the free semigroup with relations in it. Then there is a CP-semigroup $\varphi = \setind{\varphi}{s}{S}$ which does not have an E-dilation.
	\end{secTheorem}
	\begin{proof}
	We construct a product system of Hilbert spaces by taking $\sps{X}_s = \mathbb{C}$ for all $s\in \sgp{S}$, with the connecting maps given simply multiplication. We wish to define a representation for $\sps{X}$. To do this, we take the the semigroup of contractions $\setind{t}{s}{\sgp{S}}\subseteq\bdd{H}$ from Corollary \ref{example}. We also assume that $t_s\neq 0$ for all $s\in\sgp{S}$ (otherwise, we could take $t_s\oplus I_{\hilb{K}}$ for some Hilbert space $\hilb{K}$, and the resulting semigroup would still not have an isometric co-extension). A representation $T$ of $\sps{X}$ will be uniquely defined by saying who is $T_s(1)$, for all $s$. For $e$, we obviously pick $T_e(1) = I_{\hilb{H}}$, and for any other $s$ we pick $T_s(1) = t_s$.
	
	We notice that $\sps{X}_s\otimes\hilb{H} = \mathbb{C}\otimes\hilb{H}\cong\hilb{H}$, and under this identification, $\tilde{T}_s = t_s$ for all $s\neq e$. Thus $(T_e,T_s)$ is completely contractive for all $s$. The other properties we would demand from a representation are also easy to verify. We define $\varphi$ to be the CP-semigroup given by $\varphi_s(a) = \tild{T}_s\circ(I_{\mathbb{C}}\otimes a)\circ \tild{T}^*_s$. It is easy to check that $\varphi$ is in fact given by $\varphi_s(a) = t_sat^*_s$. Our aim now is to show that $\varphi$ cannot have an E-dilation.
	
	$T_s$ is injective for all $s$ (since we ensured $t_s\neq 0$), thus from Theorem 2.6 in \cite{Shalit2009}, we get that $\sps{X}$ is in fact isomorphic to the Arveson-Stinespring subproduct system associated with $\varphi$ (with $T$ corresponding to the identity representation via this isomorphism). If $\varphi$ had an E-dilation $\psi$, then by Theorem 5.12 in \cite{Shalit2009}, the Arveson-Stinespring subproduct system and identity representation $(\sps{Y'},R',\hilb{K'})$ associated with $\psi$ would be (isomorphic to) an isometric dilation of $(\sps{X},T,\hilb{H})$. But $(\sps{X},T,\hilb{H})$ cannot have an isometric dilation: if indeed $(\sps{Y},R,\hilb{K})$ is such a dilation, we may pick for any $s$, $v_s = R_s(1)$ (where by $1$, we mean here the element $1\in\mathbb{C} = \sps{X}_s\subseteq\sps{Y}_s$). It is easy to see that $\setind{v}{s}{\sgp{S}}$ would form a semigroup of isometries, co-extending $t_s$ for all $s$. Thus we reach a contradiction.
	
	\end{proof}
	
	We note that in \cite{Shalit2009}, Shalit and Solel manage to give another counter example for a 3-cycle, which also remains a counter example under rescaling. That is, a CP-semigroup $\theta = \set{\theta_n}_{n\in \mathbb{N}^3}$, such that $\lambda\theta$ has no E-dilation, for any $\lambda>0$. They do this, roughly, by exhibiting a subproduct system indexed by $\mathbb{N}^3$, which cannot be embedded in a product system, and taking the CP-semigroup induced by its shift representation. We do not know if an example which is invariant to rescaling exists for a graph with cycles strictly larger then 3, but we can see that the same approach cannot work in this case, in light of Theorem \ref{mainthm3}. That is, those subproduct systems which arise in the case of a graph with cycles strictly larger than 3, will always be embeddable in a product system.
	
\section*{Appendix} \label{appendix}
	It is our goal to give a proof for Lemma \ref{combilemma}. Let us first recall the setup. We had a graph $\gr{G}$ with no 3-cycles, and $\sgp{S}=\sgp{S}(\gr{G})$ the free unital semigroup with relations in $\gr{G}$, with alphabet $\alp{A} = V(\gr{G})$. Let $\setind{E}{a}{\alp{A}}$ be a collection of $W^*$ correspondences over some von Neumann algebra, and $\set{f_{s,t}:\set{s,t}\in E(\gr{G})}$ a unitary flip system for it. We are able, for any two presentations of a word $s=s_1\dots s_n = t_1\dots t_n$, to define a map $f:E_{s_1}\dotimes E_{s_n}\rightarrow E_{t_1}\dotimes E_{t_n}$ by composing maps of the form $I\otimes f_{a,b}\otimes I$ in some order, corresponding to some way of rewriting one presentation into the other. We will make this more precise as we attempt to prove the aforementioned lemma:
	\begin{Lemma*}
	The map $f$ defined above is unique, regardless of the order of compositions.
	\end{Lemma*}
	Before proving the lemma, we will introduce some notions to aid us, both in precisely stating what it is we wish to prove, and in proving it.
	\begin{secDef}
	For a word $s\in\sgp{S}$, we will define the graph $\gr{H}(s)$ (or just $\gr{H}$ when the identity of $s$ is clear from the context) to be the following graph: $V(\gr{H}(s))$ will be the set of presentation of $s$ (that is, just $s$, when thought of as an equivalence class of words in the free unital semigroup generated by $V(\gr{G})$). Any two presentations $v,u$ will be connected by an edge if and only if they differ by a single application of a commutation relation of $\gr{G}$. Concretely, if they are of the form $v = wabw'$, $u = wbaw'$, where $\set{a,b}\in E(\gr{G})$.
	\end{secDef}
	To any vertex $v = t_1\dots t_n\in \gr{H}$, we will couple the space $E_v:= E_{s_1}\dotimes E_{s_n}$. For every (not necessarily simple) path $p=(v_1,\dots v_n)$ in $\gr{H}$ we will define a map $f_p$, in the following way: for a 2-path $(v,u)$, supposing $v = wabw'$, $u = wbaw'$, we will define $f_{(v,u)}:= I_{E_w}\otimes f_{ab}\otimes I_{E_w'}$ (where if $w$ or $w'$ are the empty word, we remove either $I_{E_w}$ or $I_{E_w'}$ from the definition). For a general path $p=(v_1,\dots v_n)$ we define $f_p:= f_{(v_{n-1},v_n)}\circ\dots\circ f_{(v_1,v_2)}:E_{v_1}\rightarrow E_{v_n}$. For the empty path we will define $f = I$.
	
	Now the content of the lemma can be stated in a clearer fashion: for any two paths $p_1,p_2$ which share a starting and ending vertex, we have $f_{p_1} = f_{p_2}$. Taking into account that for any path $p$ and its reverse $p^r$, we have $f_{p^r} = f_{p}^* = f_{p}^{-1}$, it is easy to see that an equivalent reformulation is the following:
	\begin{secLemma}\label{apenLemma}
	For any cycle $c$ in $\gr{H}$, $f_c$ is the identity map.
	\end{secLemma}
	Before proving this, we need to note the following observation:
	\begin{secObs}\label{apendixObs}
	For two vertices $v_1,v_2$ in $\gr{H}$, if $v_1 = uw$, $v_2 = u'w$, then $u$ and $u'$ are equivalent modulo the commutation relations in $\gr{G}$. That is, their equivalence classes are equal in $\sgp{S}$ (This is really just saying that $\sgp{S}$ is a cancellative semigroup). Furthermore, the shortest path between $v_1$ and $v_2$ does not change $w$, that is, it consists only of vertices of the form $u''w$, and can be thought of as being induced by a path in $\gr{H}(u)$, connecting $u$ and $u'$.
	\end{secObs}
	This is clear. In the case of $w$ being a letter, it is easy to see that in any path between $v_1$ and $v_2$ one can just ignore the edges involving any commutation relation that has to do with $w$, and build a new, shorter path, which does not move $w$. If $w$ is not a letter we may proceed by induction.
	
	We are now ready to begin proving Lemma \ref{apenLemma}.
	\begin{proof}
	The proof will be by induction, first on the length of the word $s$, which we denote by $n$, and further on the length of cycles in $\gr{H}(s)$, which we will denote by $m$. We see that for $n=0,1$ the statement is vacuous: the graph $\gr{H}$ in those cases is one with a single vertex. That will serve as the basis for our induction on $n$. Let $s\in\sgp{S}$ be an element of length $n$, and assume that for any word of length less than $n$ the statement is true. 
	
	The basis for the induction on lengths of cycles in $\gr{H}=\gr{H}(s)$ is clear, as the lemma holds trivially for the empty path. If this is not convincing, one can see also that there are no cycles of length $1$, and that any cycle of length $2$ is an application of the same commutation relation twice, back and forth. What we mean is that it is of the form $ab\rightarrow ba\rightarrow ab$, which obviously induces the identity map. So we assume the lemma is true for cycles of length less than $m$. 
	
	Let $c = (v_1,\dots v_m,v_1)$ be a cycle in $\gr{H} = \gr{H}(s)$ of length $m$. We wish to somehow "factor" it into smaller cycles, on which we can apply the induction hypothesis. We note that if the last letter did not change throughout $c$, that is, $v_i$ is always of the form $u_ia$ for a fixed letter $a$, then it is clear that $f_c = f_{c'}\otimes I_{E_a}$, where $c'$ is a cycle in $\gr{H}(u_1)$, and thus this follows from induction on $n$. 
	
	Thus we are left with the case where the last letter of $v_1$ does change. Let us denote it by $a$. Then we may pick vertices $w_1,w_2$ to be the first instance in $c$ where $a$ is moved backwards. That is, $w_1 = uba$ and $w_2 = uab$ (where $b$ is a letter, and $u$ is a word). In the same manner, we may pick $w_3,w_4$ to be the last instance the reverse happens, so that $w_3 = u'ab'$, $w_4 = u'b'a$, and for any vertex after $w_4$, $a$ is the last letter. we give names $p_1,p_2,p_3,e_1,e_2$ to the paths connecting these vertices in $c$, in a manner best described by a diagram.
	
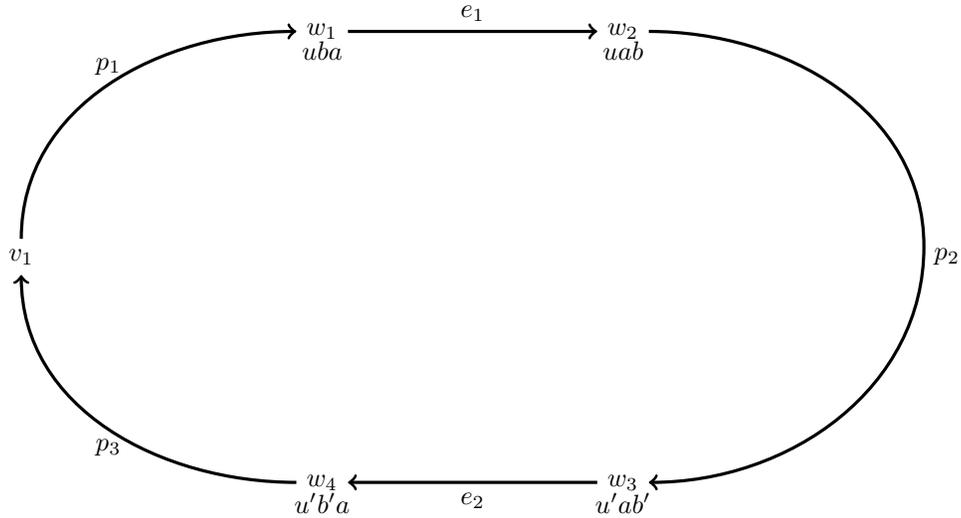
\begin{center}
	\begin{tikzpicture}[very thick,scale=2]
	\node (v1) at (0, 0) {$v_1$};
	\node (w1) at (2, 1.5) {${w_1}$};
	\node[below] (uw1) at (w1) {$uba$};
	\node (w2) at (4, 1.5) {$w_2$};
	\node[below] (uw2) at (w2) {$uab$};
	\node (w3) at (4, -1.5) {$w_3$};
	\node[below] (uw3) at (w3) {$u'ab'$};
	\node (w4) at (2, -1.5) {$w_4$};
	\node[below] (uw4) at (w4) {$u'b'a$};
	\node (aux1) at (6,0) {};
	\node[right] (aux2) at (aux1) {$p_2$};
	\draw[->] (v1)  to [out=90, in=180] node[above]{$p_1$} (w1);
	\draw[->] (w1) to  node[above]{$e_1$} (w2);
	\draw[->] (w2) to [out=0, in=90]  (aux1) to [out=-90, in=0] (w3);
	\draw[->] (w3) to node[below]{$e_2$} (w4);
	\draw[->] (w4) to [out=-180, in=-90] node[below]{$p_3$} (v1);
	\end{tikzpicture}
	\end{center}
	
	We first note that it might be the case that $w_2 = w_3$. The treatment of this case turns out to be easier, but we postpone it, and assume for now that they are different. We first claim that it must be the case that $b=b'$. Assuming otherwise, we clearly have that $a$ commutes with both $b$ and $b'$. It is also clear that $b$ and $b'$ must commute, as they must "pass" each other with a commutation relation somewhere along the path $p_2$. But then we get a 3-cycle in $\gr{G}$, which is a contradiction. 
	
	We may also note, by Observation \ref{apendixObs}, that $u$ and $u'$ are connected by a path in $\gr{H}(u)$. We take $q$ to be the minimal path from $u$ $u'$, so $q$ and $q^r$ will induce paths in $\gr{H}$: $w_4\overset{q_{4,1}}{\longrightarrow} w_1$, $w_2\overset{q_{2,3}}{\longrightarrow} w_3$. These paths have equal length (equal to the length of $q$), which is minimal, again by Observation \ref{apendixObs} and the fact that we took $q$ to be minimal. This is a good place to update our diagram with this new information.
	
	
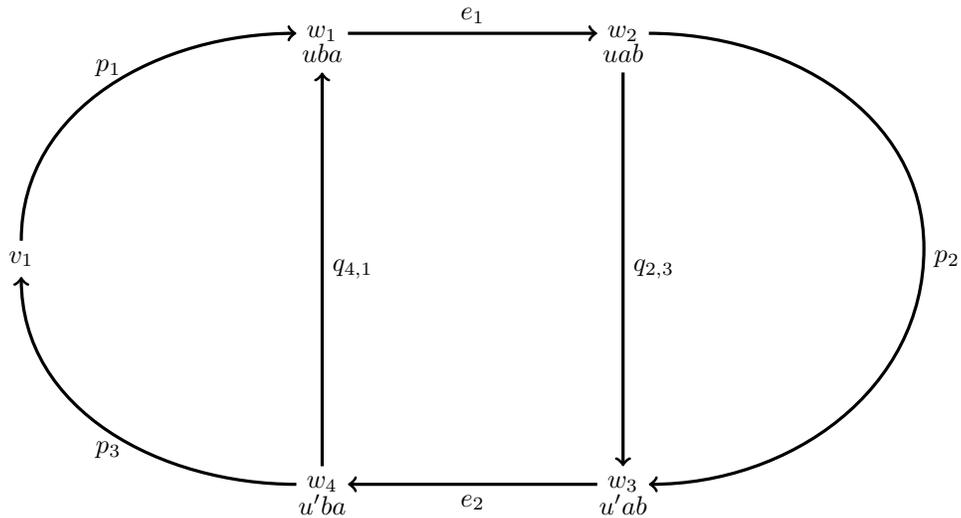
\begin{center}
	\begin{tikzpicture}[very thick,scale=2]
	\node (v1) at (0, 0) {$v_1$};
	\node (w1) at (2, 1.5) {${w_1}$};
	\node[below] (uw1) at (w1) {$uba$};
	\node (w2) at (4, 1.5) {$w_2$};
	\node[below] (uw2) at (w2) {$uab$};
	\node (w3) at (4, -1.5) {$w_3$};
	\node[below] (uw3) at (w3) {$u'ab$};
	\node (w4) at (2, -1.5) {$w_4$};
	\node[below] (uw4) at (w4) {$u'ba$};
	\node (aux1) at (6,0) {};
	\node[right] (aux2) at (aux1) {$p_2$};
	\draw[->] (v1)  to [out=90, in=180] node[above]{$p_1$} (w1);
	\draw[->] (w1) to  node[above]{$e_1$} (w2);
	\draw[->] (w2) to [out=0, in=90]  (aux1) to [out=-90, in=0] (w3);
	\draw[->] (w3) to node[below]{$e_2$} (w4);
	\draw[->] (w4) to [out=-180, in=-90] node[below]{$p_3$} (v1);
	\draw[<-] (uw1) to node[right]{$q_{4,1}$}(w4);
	\draw[<-] (w3) to node[right]{$q_{2,3}$}(uw2);
	\end{tikzpicture}
	\end{center}
	
	We get that in particular, the paths $q_{4,1}$ and $q_{2,3}$ are shorter or equal in length to $p_1\circ p_3$ and $p_2$. Thus, the cycles $(p_1,p_3,q^r_{4,1})$ and $(p_2,q^r_{2,3})$ are strictly shorter than $c$ (they have length at most $m-2$, on account of the edges $e_1,e_2$ missing), and thus induce the identity map. It is evident that if we manage to also show that the cycle $d=(e_1,q_{3,2},e_2,q_{1,4})$ induces the identity map, we will be done. We recall that we denoted by $q$ the path from $u$ to $u'$ which induced $q_{1,4}$ and $q_{3,2}$, and see that 
	\begin{align*}
	f_d &= 
	f_{q_{1,4}}\circ f_{e_2}\circ f_{q_{3,2}}\circ f_{e_1}\\
	&=
	(f_q^{-1}\otimes I_{E_{ba}})\circ 
	(I_{E_{u}}\otimes f_{b,a})\circ 
	(f_q\otimes I_{E_{ab}})\circ 
	(I_{E_{u}}\otimes f_{a,b}) \\
	&= (f_q^{-1}\circ f_q)\otimes(f_{b,a}\circ f_{ab})\\
	&=I_{E_{w_1}}.
	\end{align*}
	So we are done, and it remains only to check the case where $w_2 = w_3$, which we postponed earlier. But in this case, it is clear that we also have $w_1 = w_4$, so the appropriate diagram is:
	
	
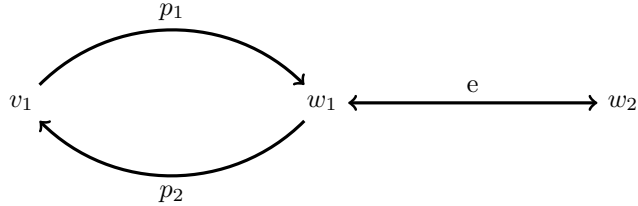
\begin{center}
	\begin{tikzpicture}[very thick,scale=2]
	\node (v1) at (0, 0) {$v_1$};
	\node (w1) at (2, 0) {${w_1}$};
	\node (w2) at (4, 0) {$w_2$};
	\draw[->] (v1)  to [out=45, in=135] node[above] {$p_1$} (w1);
	\draw[<->] (w1) to node[above] {e}(w2);
	\draw[->] (w1) to [out=-135, in=-45] node[below] {$p_2$} (v1);
	\end{tikzpicture}
	\end{center}
	where $e$ is an edge, and thus the induction hypothesis is true for the cycle $(p_1,p_2)$, and the lemma follows immediately. This concludes the proof, and the appendix.
	\end{proof}

\bibliographystyle{plain}
\bibliography{myref}
\end{document}

%% file: paperdef.tex
\usepackage{amsmath}
\usepackage{amssymb}
\usepackage{amsthm}
\usepackage{cite}
\usepackage[all,cmtip]{xy}

\newcommand {\gr}[1]{#1} 
\newcommand {\sgp}[1]{\mathcal{#1}} 
\newcommand {\vn}[1]{\mathcal{#1}} 
\newcommand {\cstar}[1]{\mathcal{#1}} 
\newcommand {\corr}[1]{#1} 
\newcommand {\hilb}[1]{#1} 
\newcommand {\sps}[1]{#1} 
\newcommand {\alp}[1]{\mathcal{#1}} 

\newcommand{\bdd}[1]{\mathcal{B}(\hilb{#1)}} 
\newcommand{\adj}[1]{\mathcal{L}(#1)} 
\newcommand{\set}[1]{\{#1\}}
\newcommand{\setind}[3]{\set{{#1}_{#2}}_{#2\in \sgp{#3}}} 

\newcommand{\dotimes}{\otimes\dots\otimes}

\newcommand{\comm}[1]{}

\newcommand{\norm}[1]{\lVert#1\rVert}
\newcommand{\restr}[2]{{
  \left.\kern-\nulldelimiterspace 
  #1 
  \vphantom{\big|} 
  \right|_{#2} 
  }}

\def \tild {\widetilde}